\documentclass[a4paper,9pt,reqno]{amsart}
\usepackage{mathtools,amsmath,amssymb}
\usepackage{hyperref}
\usepackage{enumitem}
\usepackage{cite}
\usepackage{array, booktabs}
\usepackage{multirow}
\newcolumntype{M}[1]{>{\centering\arraybackslash}m{#1}}
\hypersetup{
    colorlinks,
    citecolor=black,
    filecolor=black,
    linkcolor=black,
    urlcolor=black
}
\usepackage{tikz-cd}
\usepackage{xcolor}
\usepackage{makecell}
\usepackage{pifont}
\usepackage{graphicx}

\theoremstyle{plain}
\newtheorem{theorem}{Theorem}[section]

\newtheorem{proposition}[theorem]{Proposition}
\newtheorem{lemma}[theorem]{Lemma}

\newtheorem*{theorem*}{Theorem}

\newtheorem{innercustomgeneric}{\customgenericname}
\providecommand{\customgenericname}{}
\newcommand{\newcustomtheorem}[2]{%
  \newenvironment{#1}[1]
  {%
   \renewcommand\customgenericname{#2}%
   \renewcommand\theinnercustomgeneric{##1}%
   \innercustomgeneric
  }
  {\endinnercustomgeneric}
}

\newcustomtheorem{customthm}{Claim}

\newcommand{\RomanNumeralCaps}[1]
    {\MakeUppercase{\romannumeral #1}}

\newcommand{\cotgh}{\mathrm{cotgh}}

\theoremstyle{definition}
\newtheorem{definition}[theorem]{Definition}

\theoremstyle{remark}
\newtheorem{remark}[theorem]{Remark}

\newtheorem*{notations*}{Notations}

\title{\textbf{On constant higher order mean curvature hypersurfaces in \texorpdfstring{$\mathbb H^{\MakeLowercase{n}} \times \mathbb R$}{HnxR}}}

\author{Barbara Nelli}
\address{Universit\`a degli Studi dell'Aquila \\
Dipartimento di Ingegneria e Scienze dell'Informazione e Matematica\\
via Vetoio 1 \\
67100 L'Aquila\\
Italy}

\email{barbara.nelli@univaq.it}

\author{Giuseppe Pipoli}
\address{Universit\`a degli Studi dell'Aquila \\
Dipartimento di Ingegneria e Scienze dell'Informazione e Matematica\\
via Vetoio 1 \\
67100 L'Aquila\\
Italy}

\email{giuseppe.pipoli@univaq.it}

\author{Giovanni Russo}
\address{Florida International University \\
Department of Mathematics and Statistical Sciences \\
11200 SW 8th Street \\
FL33199, Miami \\
United States}

\email{grusso@fiu.edu}

\begin{document}

\let\thefootnote\relax

\begin{abstract}
We classify hypersurfaces with rotational symmetry and positive constant $r$-th mean curvature in $\mathbb H^n \times \mathbb R$.
Specific constant higher order mean curvature hypersurfaces invariant under hyperbolic translation are also treated.
Some of these invariant hypersurfaces are employed as barriers to prove a Ros--Rosenberg type theorem in $\mathbb H^n \times \mathbb R$:
we show that compact connected hypersurfaces of constant $r$-th mean curvature embedded in $\mathbb H^n \times [0,\infty)$ with boundary in the slice $\mathbb H^n \times \{0\}$ are topological disks under suitable assumptions.

\end{abstract}

\maketitle

\begin{center}
\emph{Dedicated to Joel Spruck}
\end{center}

\tableofcontents

\section*{Introduction}
\label{intro}
Let $M$ be a hypersurface in an $(n+1)$-dimensional Riemannian manifold and denote by $k_1,\dots,k_n$ its principal curvatures. 
The \emph{$r$-th mean curvature} of $M$ is the elementary symmetric polynomial $H_r$ in the variables $k_i$ defined as
$$\binom{n}{r}H_r \coloneqq \sum_{i_1 < \dots < i_r} k_{i_1}k_{i_2}\cdots k_{i_r}.$$
We say that $M$ is an \emph{$H_r$-hypersurface} when $H_r$ is a positive constant for some $r \in \{1,\dots,n\}$.
Note in particular that $H_1$ is the mean curvature of $M$.
In his pioneering work \cite{reilly}, Reilly showed that $H_r$-hypersurfaces in space forms appear as solutions of a variational problem, 
thus extending the corresponding property of constant mean curvature surfaces. 
Earlier, Alexandrov had dealt with higher mean curvature functions in a series of papers \cite{alexandrov}, 
and later on many existence and classification results were achieved in space forms. A list of  contributions  to this subject (far from exhaustive) is \cite{alias, barbosa, caffarelli, hounie, hsiang, leite, mori, nelli-rosenberg, nelli-semmler, nelli-zhu, oscar, ros, rosenberg, rosenberg-spruck}. 

Studies on $H_r$-hypersurfaces in more general ambient manifolds appeared in the literature more recently, see for example \cite{cheng-rosenberg, elbert-nelli2, elbert-nelli1, elbert-nelli-santos}.
Most notable for us are the results of Elbert and Sa Earp \cite{elb-earp} on $H_r$-hypersurfaces in $\mathbb H^n\times \mathbb R$, where $\mathbb H^n$ is the hyperbolic space and  de Lima--Manfio--dos Santos \cite{lima} on $H_r$-hypersurfaces in $N\times \mathbb R$, where $N$ is a Riemannian manifold.

The goal of this paper is two-fold. 
Our first result is a complete classification of rotationally invariant $H_r$-hypersurfaces in $\mathbb H^n \times \mathbb R$.  
Note that $\mathbb H^n \times \mathbb R$ has non-constant sectional curvature, but it is symmetric enough to allow a fruitful investigation of invariant hypersurfaces.
The mean curvature case $r=1$ has already been studied by Hsiang--Hsiang in \cite{hsiang} 
and B\'erard and Sa Earp \cite{berard-earp}. 
A general study of $H_r$-hypersurfaces invariant by an ambient isometry in $N \times \mathbb R$, with $N$ a Riemannian manifold, has been carried out by de Lima--Manfio--dos Santos \cite{lima}. 
We point out that part of our classification results are included in \cite{lima}, but our description and focus are different in nature for several reasons.
First, we use a parametrization that allows us to consider and analyze hypersurfaces with singularities. In fact, we get $13$ different qualitative behaviors for rotational $H_r$-hypersurfaces in $\mathbb H^n \times \mathbb R$.  
Moreover, we always include the case $n=r$, which often produces exceptional examples.
Finally, we provide detailed topological and geometric descriptions for all values of the parameters involved.

The geometry of $H_r$-hypersurfaces with $r \geq 2$ is substantially different than that of constant mean curvature hypersurfaces.
This is mainly due to the full non-linearity of the relation among the principal curvatures, in contrast with the quasi-linearity of the mean curvature equation. Most importantly, many singular cases arise and need to be classified. 
For instance, one gets conical singularities, which are not allowed in the constant mean curvature case. 
Our classification results are summarized in Tables \ref{table:1}--\ref{table:3}.

We recall that   $H_r$-hypersurfaces invariant by rotations in space forms were studied by Leite and Mori \cite{leite, mori} for the case $r=2$, and Palmas \cite{oscar} for any $r$.

Our second goal is to understand the topology of embedded $H_r$-hypersurfaces in $\mathbb H^n \times [0,\infty)$ with boundary in the horizontal slice $\mathbb H^n \times \{0\}$. We prove the following Ros--Rosenberg type theorem.
 
\begin{theorem*}
Let $M$ be a compact connected hypersurface in $\mathbb H^n \times [0,\infty)$ with constant $H_r > (n-r)/n$ and boundary in the slice $ \mathbb H^n \times \{0\}$. 
When the boundary is sufficiently small and horoconvex, then $M$ is a topological disk.
\end{theorem*}

Horoconvexity of the boundary is a natural assumption in the hyperbolic space, whereas what ``sufficiently small'' means will be explained more precisely in Section \ref{main-result}, cf.\ Theorem \ref{main-thm}.
A fundamental tool in our proof is Alexandrov reflection tecnhique, for which one needs a tangency principle. For $H_r$-hypersurfaces in Riemannian manifolds, such a tangency principle is proved by Fontenele--Silva \cite{fontenele-silva} under suitable assumptions. We point out that the geometry of our  hypersurfaces implies the existence of a strictly convex point, which guarantees the validity of the tangency principle (see Remark \ref{convex-point}).

Analogous results as in the above theorem for the constant mean curvature case are due to 
Ros--Rosenberg in $\mathbb R^3$ \cite[Theorem 1]{ros-rosenberg}, Semmler in $\mathbb H^3$ \cite[Theorem 2]{semmler}, and Nelli--Pipoli in $\mathbb H^n \times \mathbb{R}$ \cite[Theorem 4.1]{nelli-pipoli}.
For $H_r$-hypersurfaces in Euclidean space, Ros--Rosenberg theorem is proved by Nelli--Semmler \cite[Theorem 1.2]{nelli-semmler}.

In order to prove our Ros--Rosenberg type theorem we also need to discuss certain $H_r$-hypersurfaces that are invariant under hyperbolic translation. 

The structure of the paper is the following. 
In Section \ref{rotational-surfaces} we classify $H_r$-hypersurfaces in $\mathbb H^n \times \mathbb R$ with rotational symmetry.
Since the cases $r$ even and odd exhibit substantial differences, we treat them separately in two subsections.
At the end of each one, we provide complete descriptions of the various hypersurfaces that occur, see Theorems \ref{str-thm1}--\ref{str-thm4}, \ref{str-thm5}--\ref{str-thm8}, and Tables \ref{table:1}--\ref{table:3}.
In Section \ref{translation-surfaces} we study specific translation $H_r$-hypersurfaces, cf.\ Theorem \ref{str-thm-cyl2}.
Finally, in Section \ref{estimates} and \ref{limacon} we provide useful estimates and tools to be employed in Section \ref{main-result}, where we prove Ros--Rosenberg's Theorem (see Theorem~\ref{main-thm}). 

\section{Classification of rotation \texorpdfstring{$H_r$}{Hr}-hypersurfaces}
\label{rotational-surfaces}

We will generally use the Poincar\'e model of the hyperbolic space $\mathbb H^n$, $n \geq 2$.
This is defined as the open ball of Euclidean unit radius in $\mathbb R^n$ centered at the origin, and is equipped with the metric $\tilde{g}$ that at a point $x \in \mathbb H^n$ takes the form
$$\tilde{g}_x \coloneqq \left(\frac{2}{1-\lVert x\rVert^2}\right)^2(dx_1^2+\dots+dx_n^2),$$
where $\lVert \cdot \rVert$ denotes the Euclidean norm, and $(x_i)_i$ are the standard coordinates in $\mathbb R^n$. 
We work with the Riemannian cylinder $\mathbb H^n \times \mathbb R$ with product metric $g \coloneqq \tilde{g}+dt^2$,
where $t$ is a global coordinate on the $\mathbb R$ factor.

In order to describe rotational hypersurfaces inside $\mathbb H^n \times \mathbb R$ we follow \cite{elb-earp}.
Up to isometry of the ambient space, a rotationally invariant hypersurface is determined by rotation of a profile curve contained in a vertical plane through the origin inside $\mathbb H^n \times \mathbb R$.
Let us take the plane 
\begin{equation*}
V \coloneqq \{(x_1,\dots,x_n,t) \in \mathbb H^n \times \mathbb R: x_1 = \dots = x_{n-1} = 0\}, 
\end{equation*}
and consider the curve parametrized by $\rho > 0$ given as $$x_n = \tanh(\rho/2), \qquad t=\lambda(\rho).$$
The function $\lambda$ will be determined by imposing that the rotational hypersurface generated by the profile curve have $r$-th mean curvature equal to a positive constant.
We already defined the $r$-th mean curvature in the Introduction, but we write it here for further references. 
\begin{definition}
\label{r-mean-curv}
Let $k_1,\dots,k_n$ be the principal curvatures of an immersed hypersurface in any Riemannian manifold.
The \emph{$r$-th mean curvature} $H_r$ is the elementary symmetric polynomial defined as
$$\binom{n}{r}H_r \coloneqq \sum_{i_1 < \dots < i_r} k_{i_1}k_{i_2}\cdots k_{i_r}.$$
\end{definition}

Rotating the curve about the line $\{0\} \times \mathbb R$ generates a hypersurface with parametrization
$$\mathbb R_+ \times S^{n-1} \to \mathbb H^n \times \mathbb R, \qquad (\rho,\xi) \mapsto (\tanh(\rho/2)\xi,\lambda(\rho)).$$
The unit normal field to the immersion is
$$\nu = \frac{1}{(1+\dot{\lambda}^2)^{\frac12}}\left(-\frac{\dot{\lambda}}{2\cosh^2(\rho/2)}\xi,1\right),$$
and the associated principal curvatures are
\begin{equation}
\label{data}
k_1 = \dots = k_{n-1} = \cotgh(\rho)\frac{\dot{\lambda}}{(1+\dot{\lambda}^2)^{\frac12}}, \qquad k_n = \frac{\ddot{\lambda}}{(1+\dot{\lambda}^2)^{\frac32}},
\end{equation}
where $\dot{\lambda}$ denotes the derivative of $\lambda$ with respect to $\rho$. 
By applying suitable vertical reflections or translations of the hypersurface generated by the curve defined by $\lambda$, one gets several types of rotationally invariant hypersurfaces.
Care should be taken when applying the transfomation $\lambda \mapsto -\lambda$, as this changes the orientation of the hypersurface.
However, setting $\nu \mapsto -\nu$ leaves the signs of each $k_i$ unchanged.
Hereafter we classify those rotationally invariant hypersurfaces with positive constant $r$-th mean curvature.

Specializing the expression of the $r$-th mean curvature to the case $k_1 = \dots = k_{n-1}$ and $k_n$ as in \eqref{data} we find
$$nH_r = (n-r)\cotgh^r(\rho)\frac{\dot{\lambda}^r}{(1+\dot{\lambda}^2)^{\frac{r}{2}}}+\cotgh^{r-1}(\rho)\frac{r\dot{\lambda}^{r-1}\ddot{\lambda}}{(1+\dot{\lambda}^2)^{\frac{r+2}{2}}}.$$
If we divide by $\cosh^{r-1}(\rho)$ and multiply by $\sinh^{n-1}(\rho)$ both sides of the identity, we can rewrite the above as
\begin{equation}
\label{start}
n\frac{\sinh^{n-1}(\rho)}{\cosh^{r-1}(\rho)}H_r = \frac{d}{d \rho}\left(\sinh^{n-r}(\rho)\frac{\dot{\lambda}^r}{(1+\dot{\lambda}^2)^{\frac{r}{2}}} \right), \qquad r = 1,\dots,n.
\end{equation}
Choose now $H_r$ to be a positive constant, and define the function
$$I_{n,r}(\rho) \coloneqq \int_0^{\rho} \frac{\sinh^{n-1}(\tau)}{\cosh^{r-1}(\tau)}\,d\tau.$$
We can then integrate \eqref{start} once to obtain 
\begin{equation}
\label{first-integ}
nH_rI_{n,r}(\rho)+d_r = \sinh^{n-r}(\rho)\frac{\dot{\lambda}^r}{(1+\dot{\lambda}^2)^{\frac{r}{2}}},
\end{equation}
where $d_r$ is an integration constant depending on $r$. Then one integrates again to find (up to a sign for $r$ even)
\begin{equation}
\label{lambda}
\lambda_{H_r,d_r}(\rho) = \int_{\rho_-}^{\rho} \frac{(nH_rI_{n,r}(\xi)+d_r)^{\frac{1}{r}}}{\sqrt{\sinh^{\frac{2(n-r)}{r}}(\xi)-(nH_rI_{n,r}(\xi)+d_r)^{\frac{2}{r}}}}\,d\xi,
\end{equation}
where $\rho_- \geq 0$ is the infimum of the set where the integrand function makes sense. 
One should think of $\lambda$ as a one-parameter family of functions depending on $d_r$.
We write $\lambda_{H_r,d_r}$ as in \eqref{lambda} to make the dependence on $H_r$ and $d_r$ more explicit.

\begin{remark}
\label{signs}
When $r$ is even, the right-hand side in \eqref{first-integ} is non-negative, which forces the left-hand side to be non-negative as well.
In this case $-\lambda$ satisfies \eqref{first-integ}.
When $r$ is odd, identity \eqref{first-integ} implies that $\dot{\lambda}$ has the same sign of $nH_rI_{n,r}+d_r$.
Moreover, $-\lambda$ satisfies \eqref{first-integ} only after changing $\nu \mapsto -\nu$.
Lastly, critical points for $\lambda_{H_r,d_r}$ are zeros of $nH_rI_{n,r}+d_r$.
The second derivative of $\lambda_{H_r,d_r}$ is computed as
\begin{equation}
\label{second-derivative}
\ddot{\lambda}_{H_r,d_r}(\rho) = \frac{\cosh(\rho)\sinh^{\frac{2(n-r)}{r}-1}(\rho)\left(nH_r\frac{\sinh^n(\rho)}{\cosh^r(\rho)}-(n-r)(nH_rI_{n,r}(\rho)+d_r)\right)}{r(nH_rI_{n,r}(\rho)+d_r)^{\frac{r-1}{r}}\left(\sinh^{\frac{2(n-r)}{r}}(\rho)-(nH_rI_{n,r}(\rho)+d_r)^{\frac{2}{r}}\right)^{\frac{3}{2}}}.
\end{equation}
We will refer to this expression when studying the convexity of $\lambda_{H_r,d_r}$ and its regularity up to second order. 
Note that if $r>1$ the second derivative of $\lambda_{H_r,d_r}$ is not defined at its critical points.
\end{remark}

\begin{remark}
\label{integ}
Let us discuss a few more details on $I_{n,r}$. 
It is clear that $I_{n,r}(0) = 0$ and $I_{n,r}'(0) = 0$. Also, $I_{n,r}'(\rho) > 0$ and $I_{n,r}''(\rho) > 0$ for $\rho > 0$ and all $n \geq r \geq 1$, so $I_{n,r}$ is a non-negative increasing convex function.  For all values $n \geq r$ we have $nI_{n,r}(\rho) \approx \rho^n$ for $\rho \to 0$.
Moreover, for $n>r$, one has the asymptotic behavior $(n-r)I_{n,r}(\rho) \approx \sinh^{n-r}(\rho)$ for $\rho \to +\infty$, whereas for $n=r$ we have $I_{n,n}(\rho) \approx \rho$ for $\rho \to +\infty$.
\end{remark}

Next, we analyze $\lambda_{H_r,d_r}$ as in \eqref{lambda} for all values of $r = 1,\dots,n$, $H_r > 0$, and $d_r \in \mathbb R$.
The goal is to find the domain of $\lambda_{H_r,d_r}$, study its qualitative behavior, and describe the  rotational $H_r$-hypersurfaces generated by the graph of $\lambda_{H_r,d_r},$ including the description of their singularities.
This can be thought of as a classification \`a la Delaunay of rotational $H_r$-hypersurfaces in $\mathbb H^n \times \mathbb R$.
Note that we choose $n$, $r$, and $H_r > 0$ a priori, so that the family of functions $\lambda_{H_r,d_r}$ really depends only on the parameter $d_r$.
We will find a critical value of $H_r$, namely $(n-r)/n$, which we use together with the sign of $d_r$ and the parity of $r$ to distinguish various cases.
Also, we discuss $n>r$ and $n=r$ separately, as the latter case exhibits substantial differences from the former.
One may find the salient properties of the classified hypersurfaces in Tables \ref{table:1}--\ref{table:3} at the end of this section.

\subsection{Case \texorpdfstring{$r$}{r} even}
\label{case-r-even}
We start by proving the following result.

\begin{proposition}
\label{r-even-n>r1}
Assume $r$ even, $n>r$, and $d_r \leq 0$.
\begin{enumerate}
\item If $0<H_r \leq (n-r)/n$, then $\lambda_{H_r,d_r}$ is defined on $[\rho_-,+\infty)$, where $\rho_- \geq 0$ is the only solution of $nH_rI_{n,r}(\rho)+d_r = 0$.
\item If $H_r > (n-r)/n$, then $\lambda_{H_r,d_r}$ is defined on $[\rho_-,\rho_+]$, where $\rho_-$ is as above, and
$\rho_+ > 0$ is the only solution of $\sinh^{n-r}(\rho)-(nH_rI_{n,r}(\rho)+d_r)=0$.
\end{enumerate}
Further, $\lambda_{H_r,d_r}$ is increasing and convex in the interior of its domain. 
Also, $\lambda_{H_r,d_r}(\rho_-) = 0 = \lim_{\rho \to \rho_-}\dot{\lambda}_{H_r,d_r}(\rho)$.
In case $(1)$, $\lambda_{H_r,d_r}$ is unbounded. In case $(2)$ $\lim_{\rho \to \rho_+} \dot{\lambda}_{H_r,d_r}(\rho) = +\infty$. 
In both cases, $d_r= 0$ if and only if $\rho_- = 0$. We have $\lim_{\rho \to 0} \ddot{\lambda}_{H_r,0}(\rho) = H_r{}^{1/r}$, and for $d_r < 0$ one finds $\lim_{\rho \to \rho_-} \ddot{\lambda}_{H_r,d_r}(\rho) = +\infty$ (Figure \ref{fig:1}).
\end{proposition}
\begin{figure}[htb]
\begin{center}
\begin{tikzpicture}[scale=1.3]
  \draw[->] (-1.2,0) -- (2,0);
  \draw[->] (0,-0.5) -- (0,1.5);
  \draw[thick] (0.3,0) .. controls (1,0) and (1.3,0.5) .. (1.6,1);
  \draw[thick, dashed] (1.6,1) .. controls (1.72,1.2) and (1.7,1.2) .. (1.74,1.25);
  \node at (-1,1) {$H_r \leq \frac{n-r}{n}$};
  \node at (0.3,-0.2) {$\rho_-$};
  \node at (-0.15,-0.2) {$0$};
  \node at (2,-0.2) {$\rho$};
  \node at (1.2,1.2) {$\lambda_{H_r,d_r}$};
  \node at (-0.2,1.5) {$t$};
\end{tikzpicture}
\qquad 
\begin{tikzpicture}[scale=1.3]
   \draw[->] (-1.2,0) -- (2,0);
  \draw[->] (0,-0.5) -- (0,1.5);
  \draw[thick] (0.3,0) .. controls (0.5,0) and (1.5,0.01) .. (1.5,1);
  \draw[dotted] (1.5,1) -- (1.5,0);
  \node at (-1,1) {$H_r > \frac{n-r}{n}$};
  \node at (0.3,-0.2) {$\rho_-$};
  \node at (-0.15,-0.2) {$0$};
  \node at (1.5,-0.2) {$\rho_+$};
  \node at (1.15,1.2) {$\lambda_{H_r,d_r}$};
  \node at (2,-0.2) {$\rho$};
  \node at (-0.2,1.5) {$t$};
\end{tikzpicture}
\end{center}
\caption{Behavior of $\lambda_{H_r,d_r}$ for $n>r$, $r$ even, and $d_r \leq 0$. Note that $\rho_- = 0$ if and only if $d_r = 0$.}
\label{fig:1}
\end{figure}
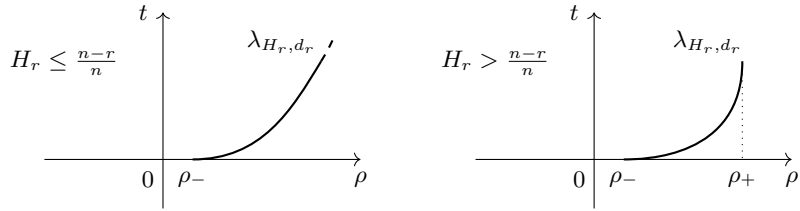
\begin{proof}
The function $nH_rI_{n,r}+d_r$ must be non-negative as noted in Remark~\ref{signs}, hence $\lambda_{H_r,d_r}$ is well-defined when
$$0 \leq nH_rI_{n,r}(\rho)+d_r < \sinh^{n-r}(\rho).$$
There is a unique value $\rho_- \geq 0$ depending on $d_r$ such that $nH_rI_{n,r}(\rho_-)+d_r = 0$, and Remark \ref{integ} implies $d_r = 0$ if and only if $\rho_-=0$.
Set 
$$f(\rho) \coloneqq \sinh^{n-r}(\rho)-(nH_rI_{n,r}(\rho)+d_r), \qquad \rho \geq 0.$$
Then $f(\rho_-) \geq 0$ and $f'(\rho) = \sinh^{n-r-1}(\rho)\cosh(\rho)((n-r)-nH_r\tanh^r(\rho))$.
We have $f'(\rho) > 0$ for $\rho > \rho_-$ when $\tanh^r(\rho) < (n-r)/nH_r$. So if $0<H_r \leq (n-r)/n$ the inequality is always true, and $f$ has no zeros in $(\rho_-,+\infty)$.
If $H_r > (n-r)/n$ then $\lim_{\rho \to +\infty} f'(\rho) = -\infty$, so $f$ eventually decreases to $-\infty$.
This implies $f$ has a zero $\rho_+ > \rho_-$ depending on the value of $d_r$.

It follows that $\lambda_{H_r,d_r}$ is defined on some interval with $\rho_-$ as minimum.
If $0<H_r \leq (n-r)/n$ then the interval is unbounded.
We have $\lambda_{H_r,d_r}(\rho_-) = 0 = \lim_{\rho \to \rho_-} \dot{\lambda}_{H_r,d_r}(\rho)$, and $\lim_{\rho \to +\infty} \lambda_{H_r,d_r}(\rho) = +\infty$ by the asymptotic behavior of $I_{n,r}$ noted in Remark \ref{integ}.
Moreover, $\lambda_{H_r,d_r}$ is increasing as the integrand function is positive away from $\rho_-$.
If $H_r > (n-r)/n$ then the denominator of the integrand function has a zero $\rho_+$ depending on $d_r$. 
This means $\lambda_{H_r,d_r}$ is defined on $[\rho_-,\rho_+)$, and its slope tends to $+\infty$ when $\rho \to \rho_+$. 
We claim that $\lambda_{H_r,d_r}$ is finite at $\rho_+$.
Convergence of the integral is essentially determined by the behavior of 
$$h(\rho) \coloneqq \sinh^{\frac{n-r}{r}}(\rho)-(nH_rI_{n,r}(\rho)+d_r)^{\frac{1}{r}}$$ 
near $\rho_+$.
But $h(\rho_+) = 0$, and $h'(\rho_+)$ is finite, which implies that $\lambda_{H_r,d_r}$ behaves as the integral of $1/(\rho_+-\rho)^{1/2}$ for $\rho$ close to $\rho_+$, whence convergence at $\rho_+$.

In order to check convexity on $(\rho_-,\rho_+)$, observe that the sign of $\ddot{\lambda}_{H_r,d_r}$ as in \eqref{second-derivative} is determined by the sign of
$$g(\rho) \coloneqq \frac{\sinh^n(\rho)}{\cosh^r(\rho)}-(n-r)I_{n,r}(\rho) - \frac{d_r(n-r)}{nH_r}.$$
We trivially have $g(\rho_-) \geq 0$ and $g'(\rho) = r\sinh^{n-1}(\rho)/\cosh^{r+1}(\rho) > 0$, so that $g(\rho)$ is always positive for $\rho > 0$.
Continuity of the second derivative of $\lambda_{H_r,d_r}$ at the origin for $d_r=0$ follows by an explicit calculation using Remark \ref{integ},
whereas the statement $\lim_{\rho \to \rho_-} \ddot{\lambda}_{H_r,d_r}(\rho) = \infty$ for $d_r < 0$ is trivial, cf.\ \eqref{second-derivative}.
\end{proof}

We now go on with the analysis of the case $d_r > 0$, but we first make a few technical considerations.
For $r > 2$ we have the following formula, which can be proved via integration by parts:
\begin{equation}
\label{integral_formula}
I_{r+1,r}(x) = -\frac{\sinh^{r-1}(x)}{(r-2)\cosh^{r-2}(x)}+\frac{r-1}{r-2}I_{r-1,r-2}(x).
\end{equation} 
Recall that for a natural number $m$ the double factorial is $m!! \coloneqq m(m-2)!!$, and $1!! = 0!! = 1$.
Now take $r > 2$ even. 
From the recurrence relation \eqref{integral_formula} we derive the following closed expression for $I_{r+1,r}(x)$:
\begin{align}
\label{formula_integral}
I_{r+1,r}(x) & = -\sinh(x)\biggl(\frac{1}{r-2}\tanh^{r-2}(x)+\frac{r-1}{(r-2)(r-4)}\tanh^{r-4}(x) \nonumber \\
& \qquad +\frac{(r-1)(r-3)}{(r-2)(r-4)(r-6)}\tanh^{r-6}(x)+\dots + \frac{(r-1)!!}{3(r-2)!!}\tanh^2(x)\biggr) \nonumber \\ 
& \qquad + \frac{(r-1)!!}{(r-2)!!}I_{3,2}(x).
\end{align}
The explicit expression $I_{3,2}(x) = \sinh(x)-\arctan(\sinh(x))$ returns now a closed formula for each $I_{r+1,r}(x)$. 
We note here a useful identity which can be proved by induction.

\begin{lemma}
\label{important_identity}
Let $r \geq 2$ be an even natural number. Then
\begin{align*}
\frac{(r-1)!!}{(r-2)!!} & = 1+\frac{1}{r-2}+\frac{r-1}{(r-2)(r-4)}+\frac{(r-1)(r-3)}{(r-2)(r-4)(r-6)} +\\
& \qquad + \dots + \frac{(r-1)!!}{3(r-2)!!},
\end{align*}
where, for all $r,$ the sum on the right-hand side must be truncated in such a way that all summands exist.
\end{lemma}
We shall see that when $d_r > 0$ then $\lambda_{H_r,d_r}$ is not well-defined for $d_r$ too large. 
We will combine \eqref{formula_integral} and Lemma \ref{important_identity} to give a precise upper bound for $d_r$ when $n=r+1$ and $H_r = (n-r)/n = 1/(r+1)$.
\begin{proposition}
\label{r_even-2}
Assume $r$ even, $n>r$, and $d_r > 0$.
\begin{enumerate}
\item If $0<H_r < (n-r)/n$, then $\lambda_{H_r,d_r}$ is defined on $[\rho_-,+\infty)$, where $\rho_- > 0$ is the only solution of $\sinh^{n-r}(\rho)-(nH_rI_{n,r}(\rho)+d_r) = 0$ on $(0,\infty)$. 
\item If $H_r = (n-r)/n$, then when $n=r+1$ we need $d_r < (r-1)!!\pi/2(r-2)!!$ for $\lambda_{H_r,d_r}$ to be well-defined, whereas for $n>r+1$ we have no constraint.
Under such conditions, the results in the previous point hold.
\item If $H_r > (n-r)/n$, set $\tau > 0$ such that $\tanh^r(\tau) = (n-r)/nH_r$. Then $d_r < \sinh^{n-r}(\tau)-nH_rI_{n,r}(\tau)$ for $\lambda_{H_r,d_r}$ to
be defined. So $\lambda_{H_r,d_r}$ is a function on $[\rho_-,\rho_+] \subset (0,+\infty)$, where $\sinh^{n-r}(\rho_{\pm})-(nH_rI_{n,r}(\rho_{\pm})+d_r) = 0$. 
\end{enumerate}
Further, $\lambda_{H_r,d_r}$ is increasing in the interior of its domain.
In cases $(1)$--$(2)$, $\lambda_{H_r,d_r}(\rho_-) = 0$, $\lim_{\rho \to \rho_-} \dot{\lambda}_{H_r,d_r}(\rho) = +\infty$, $\lambda_{H_r,d_r}$ is unbounded, and is concave in the interior of its domain.
In case $(3)$, $\lambda_{H_r,d_r}(\rho_-) = 0$, $\lim_{\rho \to \rho_{\pm}}\dot{\lambda}_{H_r,d_r}(\rho) = +\infty$, $\lambda_{H_r,d_r}$ has a unique inflection point in $(\rho_-,\rho_+)$, and goes from being concave to convex (Figure \ref{fig:2}).
\end{proposition}
\begin{figure}[htb]
\begin{center}
\begin{tikzpicture}[scale=1.3]
  \draw[->] (-1.2,0) -- (2,0);
  \draw[->] (0,-0.5) -- (0,1.5);
  \draw[thick] (0.3,0) .. controls (0.3,0.5) and (1,1) .. (1.5,1.3);
  \draw[thick,dashed] (1.5,1.3) .. controls (1.7,1.42) and (1.7,1.41) .. (1.75,1.44);
  \node at (-1,1) {$H_r \leq \frac{n-r}{n}$};
  \node at (0.3,-0.2) {$\rho_-$};
  \node at (-0.15,-0.2) {$0$};
  \node at (2,-0.2) {$\rho$};
  \node at (2,1.2) {$\lambda_{H_r,d_r}$};
  \node at (-0.2,1.5) {$t$};
\end{tikzpicture}
\qquad 
\begin{tikzpicture}[scale=1.3]
   \draw[->] (-1.2,0) -- (2,0);
  \draw[->] (0,-0.5) -- (0,1.5);
  \draw[thick] (0.3,0) .. controls (0.3,0.6) and (1.5,0.5) .. (1.5,1);
  \draw[dotted] (1.5,1) -- (1.5,0);
  \node at (-1,1) {$H_r > \frac{n-r}{n}$};
  \node at (0.3,-0.2) {$\rho_-$};
  \node at (-0.15,-0.2) {$0$};
  \node at (1.5,-0.2) {$\rho_+$};
  \node at (1.15,1.2) {$\lambda_{H_r,d_r}$};
  \node at (2,-0.2) {$\rho$};
  \node at (-0.2,1.5) {$t$};
\end{tikzpicture}
\end{center}
\caption{Behavior of $\lambda_{H_r,d_r}$ for $n>r$, $r$ even, and $d_r > 0$.}
\label{fig:2}
\end{figure}
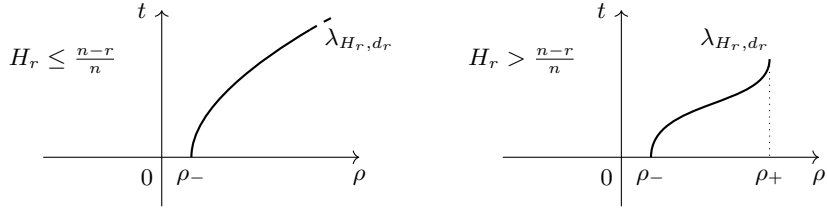
\begin{proof}
We have the constraint $0 \leq nH_rI_{n,r}(\rho)+d_r < \sinh^{n-r}(\rho)$ for $\rho > 0$.
Since $I_{n,r}(0)=0$ and $d_r > 0$ we must have $\rho_- > 0$. Such a $\rho_-$ exists only if
$$f(\rho) \coloneqq \sinh^{n-r}(\rho)-(nH_rI_{n,r}(\rho)+d_r)$$
has a zero. We have $f(0) < 0$ and
$$f'(\rho) = \sinh^{n-r-1}(\rho)\cosh(\rho)\left((n-r)-nH_r\tanh^r(\rho)\right).$$
For $0<H_r < (n-r)/n$ the derivative $f'$ is always positive and tends to $+\infty$ as $\rho$ runs to $\infty$,
so $\rho_-$ exists and $\lambda_{H_r,d_r}$ is defined on $[\rho_-,+\infty)$. For $H_r = (n-r)/n$ we have a more subtle behavior. We compute
\begin{align*}
\frac{1}{n-r}\lim_{\rho \to \infty} f'(\rho) & = \lim_{\rho\to\infty} \sinh^{n-r-1}(\rho)\cosh(\rho)(1-\tanh^r(\rho)) \\
& = \lim_{\rho\to\infty} \sinh^{n-r-1}(\rho)\frac{\cosh^r(\rho)-\sinh^r(\rho)}{\cosh^{r-1}(\rho)} \\
& = \lim_{\rho\to\infty} \sinh^{n-r-1}(\rho)\frac{(\cosh(\rho)-\sinh(\rho))\sum_{i=0}^{r-1} \cosh^{r-1-i}(\rho)\sinh^{i}(\rho)}{\cosh^{r-1}(\rho)} \\
& = \lim_{\rho\to\infty} \frac{\sinh^{n-r-1}(\rho)}{\cosh(\rho)+\sinh(\rho)}\sum_{i=0}^{r-1} \tanh^i(\rho).
\end{align*} 
When $n=r+2$ the limit of $f'$ is $r$, and if $n>r+2$ the limit is $+\infty$. 
In these two cases $\rho_-$ exists and $\lambda_{H_r,d_r}$ is defined on $[\rho_-,\infty)$.
The case $n=r+1$ needs to be studied separately, as the limit vanishes. 
The claim is that for any $r$ even we have that $\rho_-$ exists only if $$d_r < \frac{(r-1)!!}{(r-2)!!}\frac{\pi}{2}.$$
Indeed, when $r=2$ we compute
\begin{equation*}
\lim_{\rho\to +\infty} \sinh(\rho)-\int_0^{\rho} \frac{\sinh^2(\sigma)}{\cosh(\sigma)}\,d\sigma -d_2 = \lim_{\rho\to +\infty} (\arctan(\sinh(\rho))-d_2) = \tfrac{\pi}{2}-d_2.
\end{equation*}
In this case, $f$ cannot have a zero if $d_2 \geq \pi/2$. 
To prove the above claim for $r\geq4$, we use \eqref{formula_integral} and find
\begin{align*}
f(\rho) & = \sinh(\rho)-I_{r+1,r}(\rho)-d_r \\
& = \sinh(\rho)\biggl(1+\frac{1}{r-2}\tanh^{r-2}(\rho)+\frac{r-1}{(r-2)(r-4)}\tanh^{r-4}(\rho) \\
& \quad +\dots+\frac{(r-1)(r-3)\cdots 5}{(r-2)(r-4)\cdots 2}\tanh^2(\rho)-\frac{(r-1)!!}{(r-2)!!}\biggr) \\
& \quad +\frac{(r-1)!!}{(r-2)!!}\arctan(\sinh(\rho))-d_r.
\end{align*}
Now Lemma \ref{important_identity} implies that when $\rho \to +\infty$ the sum of the terms into brackets goes to zero,
and the product of $\sinh(\rho)$ with the latter vanishes (one can use the estimates $\sinh(\rho) \approx e^\rho/2$ and $\tanh(\rho) \approx 1-2e^{-2\rho}$ for $\rho \to +\infty$ to see this).
Hence $$\lim_{\rho \to +\infty} f(\rho) = \frac{(r-1)!!}{(r-2)!!}\frac{\pi}{2}-d_r,$$
and the claim is proved. 
Convergence of $\lambda_{H_r,d_r}$ at $\rho_-$ follows by a similar argument as in the proof of Proposition \ref{r-even-n>r1}.

If $H_r > (n-r)/n$ there is a $\tau > 0$ such that $f$ is 
increasing on $(0,\tau)$ and decreasing on $(\tau,+\infty)$. Such a $\tau$ satisfies $\tanh^r(\tau) = (n-r)/nH_r$.
In order to have a well-defined $\lambda_{H_r,d_r}$, we necessarily want $f(\tau) > 0$, which forces the condition
$$d_r < \sinh^{n-r}(\tau)-nH_rI_{n,r}(\tau).$$
Since $f'(\rho_-) > 0$, $f'(\rho_+) < 0$, then $f$ vanishes at $\rho_-$ and $\rho_+$ with order $1$. 
This gives convergence of $\lambda_{H_r,d_r}$ at the boundary points. We have $\lambda_{H_r,d_r}(\rho_-) = 0$,
$\lambda_{H_r,d_r}(\rho_+) > 0$, and $\lim_{\rho \to \rho_{\pm}} \dot{\lambda}_{H_r,d_r}(\rho) = +\infty$ at once.

We finally discuss convexity of $\lambda_{H_r,d_r}$ by proceeding as in the case $d_r \leq 0$.
The sign of the second derivative is determined by the sign of 
$$g(\rho) \coloneqq \frac{\sinh^n(\rho)}{\cosh^r(\rho)}-(n-r)I_{n,r}(\rho) - \frac{d_r(n-r)}{nH_r}.$$
By definition of $\rho_-$, the sign of $g(\rho_-)$ is determined by the sign of $\tanh^r(\rho_-)-(n-r)/nH_r$.
When $nH_r > n-r$, then the above quantity is negative as 
$$\tanh^r(\rho_-)-\frac{n-r}{nH_r} = \tanh^r(\rho_-) - \tanh^r(\tau).$$
Similarly, $g(\rho_+) > 0$. Since $g'(\rho) > 0$, $\lambda_{H_r,d_r}$ has a unique inflection point, and goes from being concave to convex.
If $nH_r \leq n-r$, we have $\lim_{\rho \to +\infty} g(\rho) = -d_r(n-r)/nH_r < 0$ by Remark \ref{integ}. 
But $g$ is an increasing function, so it is always negative, and hence $\lambda_{H_r,d_r}$ is concave.
\end{proof}

There remains to look at the case $n=r$. Set $I_n(\rho) \coloneqq I_{n,n}(\rho) = \int_0^{\rho} \tanh^{n-1}(\tau)\,d\tau$.

\begin{proposition}
\label{r-even-n=r}
Assume $n=r$ even. Then $\lambda_{H_n,d_n}$ is well-defined for $d_n < 1$. 
\begin{enumerate}
\item If $d_n < 0$, then $\lambda_{H_n,d_n}$ is defined on $[\rho_-,\rho_+]$, where
$\rho_-$ is the only solution of $nH_nI_n(\rho)+d_n = 0$, and $\rho_+$ is the only solution of $nH_nI_n(\rho)+d_n = 1$. 
\item If $0 \leq d_n < 1$, then $\lambda_{H_n,d_n}$ is defined on $[0,\rho_+]$, where $\rho_+$
is defined as above. 
\end{enumerate}
Further, $\lambda_{H_n,d_n}$ is increasing and convex in the interior of its domain.
In case $(1)$, $\lambda_{H_n,d_n}(\rho_-) = 0 = \dot{\lambda}_{H_n,d_n}(\rho_-)$, and $\lim_{\rho \to \rho_+} \dot{\lambda}_{H_n,d_n}(\rho) = +\infty$.
In case $(2)$, $\lambda_{H_n,d_n}(0) = 0$, $\dot{\lambda}_{H_n,d_n}(\rho_-) = d_n^{1/n}/(1-d_n^{2/n})^{1/2}$, and $\lim_{\rho \to \rho_+} \dot{\lambda}_{H_n,d_n}(\rho) = +\infty$.
In the particular case $d_n=0$, we also have $\lim_{\rho \to 0} \ddot{\lambda}_{H_n,0}(\rho) = H_n{}^{1/n}$, and if $d_n < 0$ then $\lim_{\rho \to \rho_-} \ddot{\lambda}_{H_r,d_r}(\rho_-) = +\infty$ (Figure \ref{fig:3}).
\end{proposition}
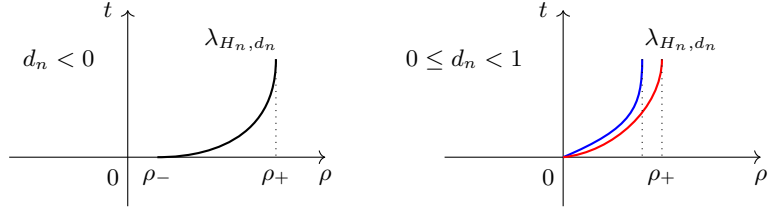
\begin{figure}[htb]
\begin{center}
\begin{tikzpicture}[scale=1.3]
   \draw[->] (-1.2,0) -- (2,0);
  \draw[->] (0,-0.5) -- (0,1.5);
  \draw[thick] (0.3,0) .. controls (0.5,0) and (1.5,0.01) .. (1.5,1);
  \draw[dotted] (1.5,1) -- (1.5,0);
  \node at (-0.7,1) {$d_n < 0$};
  \node at (0.3,-0.2) {$\rho_-$};
  \node at (-0.15,-0.2) {$0$};
  \node at (1.5,-0.2) {$\rho_+$};
  \node at (1.15,1.2) {$\lambda_{H_n,d_n}$};
  \node at (2,-0.2) {$\rho$};
  \node at (-0.2,1.5) {$t$};
\end{tikzpicture}
\qquad 
\begin{tikzpicture}[scale=1.3]
  \draw[->] (-1.2,0) -- (2,0);
  \draw[->] (0,-0.5) -- (0,1.5);
  \draw[thick,blue] (0,0) .. controls (0.7,0.3) and (0.8,0.5) .. (0.8,1);
  \draw[thick,red] (0,0) .. controls (0.5,0.05) and (1,0.5) .. (1,1);
  \draw[dotted] (1,1) -- (1,0);
   \draw[dotted] (0.8,1) -- (0.8,0);
  \node at (-1,1) {$0 \leq d_n < 1$};
  \node at (-0.15,-0.2) {$0$};
  \node at (1,-0.2) {$\rho_+$};
  \node at (1.2,1.2) {$\lambda_{H_n,d_n}$};
  \node at (2,-0.2) {$\rho$};
  \node at (-0.2,1.5) {$t$};
\end{tikzpicture}
\end{center}
\caption{Behavior of $\lambda_{H_n,d_n}$ for $n$ even and $H_n > 0$.
When $d_n$ is non-negative, we distinguish two cases, i.e.\ $d_n = 0$ (red), and $0 < d_n < 1$ (blue).}
\label{fig:3}
\end{figure}
\begin{proof}
Our usual constraint becomes
$$0 \leq nH_nI_n(\rho)+d_n < 1.$$
Hence necessarily $d_n < 1$. If $d_n < 0$ there are positive numbers $\rho_-,\rho_+$ such that $nH_nI_n(\rho_-) +d_n=0$ and $nH_nI_n(\rho_+)+d_n = 1$,
and $\lambda_{H_n,d_n}$ is defined on $[\rho_-,\rho_+)$. Clearly $\dot{\lambda}_{H_n,d_n}(\rho_-) = 0$. 
If $0 \leq d_n < 1$, then $\lambda_{H_n,d_n}$ is defined on $[0,\rho_+)$.
We have $\dot{\lambda}_{H_n,d_n}(0) = d_n^{1/n}/(1-d_n^{2/n})^{1/2}$. 
The same method as in the proof of Proposition~\ref{r-even-n>r1} shows that in both cases  $\lambda_{H_n,d_n}$ is finite at  $\rho_+$.
The expression of $\ddot{\lambda}_{H_r,d_r}$ in \eqref{second-derivative} for $n=r$ implies convexity of the graphs at once.
Continuity of the second derivative at the origin for $d_n=0$ follows by an explicit calculation, cf.\ \eqref{second-derivative} and Remark \ref{integ}.
\end{proof}

We now study the regularity of the $H_r$-hypersurface generated by rotating the graph of $\lambda_{H_r,d_r}$, as described at the beginning of Section \ref{rotational-surfaces}.
Then we will proceed with the classification result.

\begin{proposition}
\label{prop:c2-reg}
Let $n\geq r$, $r$ even. 
Then the hypersurface generated by the curve defined by $\lambda_{H_r,d_r}$ is of class $C^2$ at $\rho = \rho_+$, when the latter exists, and it is of class $C^2$ at $\rho=\rho_-$ if and only if $n>r$ and $d_r\geq 0$ or $n=r$ and $d_n=0$. 
When $n=r$ and $d_n>0$, it has a conical singularity at $\rho=0$. If $n\geq r$ and $d_r < 0$, it has cuspidal singularities at $\rho=\rho_-$.
\end{proposition}
\begin{proof}
Regularity to second order of the hypersurface generated by the graph of $\lambda_{H_r,d_r}$ is proved by showing that the second fundamental form $A$ is bounded.

For any choice of $n\geq r$, $H_r$ and $d_r$ for which $\rho_+$ exists, we have that $\rho_+>0$ and $\lim_{\rho\rightarrow\rho_+}\dot{\lambda}_{H_r,d_r}(\rho)=+\infty$.
By \eqref{data}, for any $i=1,\dots,n-1$ we have that
$$\lim_{\rho\rightarrow\rho_+}k_i(\rho)=\cotgh(\rho_+).$$
By definition of $\rho_+$, combining \eqref{data} and \eqref{second-derivative} one finds
$$\lim_{\rho\rightarrow\rho_+}k_n(\rho)=\frac{\cotgh(\rho_+)}{r}(nH_r\tanh(\rho_+)-(n-r)).$$
It follows that $\lim_{\rho\rightarrow\rho_+}|A|^2(\rho)$ exists and is finite.

Assume now that $n>r$ and $d_r>0$, then $\rho_->0$ and $\lim_{\rho\rightarrow\rho_-}\dot{\lambda}_{H_r,d_r}(\rho)=+\infty$.
Therefore $\lim_{\rho\rightarrow\rho_-}|A|^2(\rho)$ exists and is finite by arguing as above.

When $d_r=0$ we have $\rho_-=0$. By Remark \ref{integ}, \eqref{data} and \eqref{second-derivative}, as $\rho\rightarrow 0$ we get the estimates
\begin{equation*}
\cotgh(\rho) \approx \rho^{-1}, \qquad
\dot{\lambda}_{H_r,d_r}(\rho) \approx H_r^{\frac 1r}\rho, \qquad
\ddot{\lambda}_{H_r,d_r}(\rho) \approx H_r^{\frac 1r}.
\end{equation*}
For any $i=1,\dots,n$ it follows that
$$
\lim_{\rho\rightarrow 0}k_i(\rho)=H_r^{\frac 1r},
$$
and $\lim_{\rho\rightarrow0}|A|^2(\rho)$ exists and is finite in this case as well.

In the case $n\geq r$ and $d_r<0$ we have  $\rho_->0$, $\dot{\lambda}_{H_r,d_r}(\rho_-)=0$, but $\lim_{\rho\rightarrow\rho_-}\ddot{\lambda}_{H_r,d_r}(\rho)=+\infty$. Hence $|A|^2$ blows up at $\rho_-$ because $\lim_{\rho\rightarrow\rho_-}k_n(\rho)=+\infty$. 
Moreover, it is clear that by reflecting the hypersurface generated by the graph of $\lambda_{H_r,d_r}$ across the slice ${\mathbb H}^n\times\{0\}$ one gets cuspidal singularities along the intersection with ${\mathbb H}^n\times\{0\}.$

Finally, when $n=r$ and $0<d_n<1$, by Proposition \ref{r-even-n=r} we have that $\rho_-=0$ and $$\dot\lambda_{H_n,d_n}(0)=\frac{d_n^{\frac 1n}}{(1-d_n^{\frac 2n})^{\frac 12}}>0.$$
So the hypersurface generated by the graph of $\lambda_{H_n,d_n}$ has a  conical singularity in $\rho=0$.
\end{proof}

We now classify rotational $H_r$-hypersurfaces for $r$ even based on the above arguments. 
We recover results by Elbert--Sa Earp \cite[Section 6]{elb-earp} and de Lima--Manfio--dos Santos \cite[Theorem 1 and 2]{lima}.
We recall that a \emph{slice} is  any subspace $\mathbb H^n\times\{t\} \subset \mathbb H^n \times \mathbb R$, and by its \emph{origin} we mean its intersection with the $t$-axis. 

\begin{theorem}
\label{str-thm1}
Assume $r$ even, $n>r$, and $d_r < 0$. 
By reflecting the rotational hypersurface given by the graph of $\lambda_{H_r,d_r}$ across suitable  slices,  we get a  non-compact embedded $H_r$-hypersurface.

\begin{enumerate}
\item If $0<H_r \leq (n-r)/n$, the hypersurface generated by the graph of $\lambda_{H_r,d_r}$ together with its reflection across the slice ${\mathbb H}^n\times\{0\}$ is a singular annulus. 
Its singular set is made of cuspidal points along a sphere of radius $\rho_-$ centered at the origin of the slice $\mathbb H^n \times \{0\}.$
\item If $H_r > (n-r)/n$, then the hypersurface generated by the graph of $\lambda_{H_r,d_r}$, together with its reflections across the  slices $\mathbb H^n \times \{k\lambda_{H_r,d_r}(\rho_+)\}$, $k \in \mathbb Z$, gives a singular onduloid.
Its singular set is made of cuspidal points along spheres of radius $\rho_-$ centered at the origin of the slices $\mathbb H^n \times \{2k\lambda_{H_r,d_r}(\rho_+)\}$, $k \in \mathbb Z$.
\end{enumerate}
\end{theorem}
\begin{theorem}
\label{str-thm2}
Assume $r$ even, $n>r$, and $d_r = 0$. 
Then the rotational hypersurface given by the graph of $\lambda_{H_r,0}$ is a complete embedded $H_r$-hypersurface, possibly after reflection across a suitable slice.
\begin{enumerate}
\item If $0<H_r \leq (n-r)/n$, the hypersurface generated by the graph of $\lambda_{H_r,0}$  is an  entire graph of class $C^2$ tangent to the  slice $\mathbb H^n \times \{0\}$ at the origin.
\item If $H_r > (n-r)/n$, the hypersurface generated by the graph of $\lambda_{H_r,0}$, together with its reflection across the slice $\mathbb H^n \times \{\lambda_{H_r,0}(\rho_+)\}$, is a class $C^2$ sphere.
\end{enumerate}
\end{theorem}
\begin{theorem}
\label{str-thm3}
Assume $r$ even, $n>r$, and $d_r > 0$. 
By reflecting the rotational hypersurface given by the graph of $\lambda_{H_r,d_r}$ across suitable slices, we get a complete non-compact embedded $H_r$-hypersurface.
\begin{enumerate}
\item If $0<H_r \leq (n-r)/n$, the hypersurface generated by the graph of $\lambda_{H_r,d_r}$, together with 
its reflection across the slice ${\mathbb H}^n\times\{0\}$, is a class $C^2$ annulus. When $n=r+1$ and $H_r = 1/(r+1)$, the same holds, provided that $d_r$ is smaller than $(r-1)!!\pi/2(r-2)!!$.
\item If $H_r > (n-r)/n$, the hypersurface generated by the graph of $\lambda_{H_r,d_r}$ together with its reflections across the  slices $\mathbb H^n \times \{k\lambda_{H_r,d_r}(\rho_+)\}$, $k \in \mathbb Z$, is a class $C^2$ onduloid.
\end{enumerate}
\end{theorem}
\begin{theorem}
\label{str-thm4}
Assume $n=r$ even and $H_n > 0$. Then the $H_n$-hypersurface generated by the graph of $\lambda_{H_n,d_n},$ together with its reflection across the slice $\mathbb H^n \times \{\lambda_{H_n,d_n}(\rho_+)\}$, is a class $C^2$ sphere if $d_n=0$, and a peaked sphere if 
$0 < d_n < 1$. If $d_n < 0$ then the $H_n$-hypersurface generated by the graph of $\lambda_{H_n,d_n}$, together with its reflections across the slices $\mathbb H^n \times \{k\lambda_{H_n,d_n}(\rho_+)\}$, $k \in \mathbb Z$, gives a singular onduloid.
Its singular set is made of cuspidal points along spheres of radius $\rho_-$ centered at the origin of the slices $\mathbb H^n \times \{2k\lambda_{H_n,d_n}(\rho_+)\}$, $k \in \mathbb Z$.
\end{theorem}

\subsection{Case \texorpdfstring{$r$}{r} odd}
\label{case-r-odd}
We organize this subsection in a similar fashion as the previous one.
Some of the arguments will be analogous to the corresponding ones for $r$ even, so we leave out the relative details.
Note that this subsection includes and extends the mean curvature case treated in \cite{berard-earp} and \cite{nelli-pipoli}.
A crucial difference from the case $r$ even is that for $d_r < 0$ the derivative $\dot{\lambda}_{H_r,d_r}$ is negative on some subset of the domain of $\lambda_{H_r,d_r}$, and for $r > 1$ the function $\lambda_{H_r,d_r}$ is not $C^2$-regular at its minimum point.
Further, more types of curves arise when $n > r$ and $d_r < 0$, and when $n=r$.
In our classification, we will recover results by B\'erard--Sa Earp \cite[Section 2]{berard-earp}, Elbert--Sa Earp \cite[Section 6]{elb-earp}, and de Lima--Manfio--dos Santos \cite[Theorem 1 and 2]{lima}.

\begin{proposition}
\label{r-odd-n>r}
Assume $r$ odd, $n>r$, and $d_r < 0$.
\begin{enumerate}
\item If $0<H_r \leq (n-r)/n$, then $\lambda_{H_r,d_r}$ is defined on $[\rho_-,+\infty)$, where $\rho_- > 0$ is the only solution of $\sinh^{n-r}(\rho)+(nH_rI_{n,r}(\rho)+d_r) = 0$. 
\item If $H_r > (n-r)/n$, then $\lambda_{H_r,d_r}$ is defined on $[\rho_-,\rho_+]$, where $\rho_-$ is as above, and $\rho_+>0$ is the only solution of $\sinh^{n-r}(\rho)-(nH_rI_{n,r}(\rho)+d_r)=0$.
\end{enumerate}
Set $\rho_0$ to be the only zero of $nH_rI_{n,r}+d_r$. 
We have $\lambda_{H_r,d_r}(\rho_-) = 0$, $\lim_{\rho \to \rho_-} \dot{\lambda}_{H_r,d_r}(\rho) = -\infty$, $\dot{\lambda}_{H_r,d_r}(\rho) < 0$ when $\rho_- < \rho < \rho_0$, 
and $\dot{\lambda}_{H_r,d_r}(\rho) > 0$ when $\rho > \rho_0$. In case $(1)$, $\lim_{\rho \to +\infty} \lambda_{H_r,d_r}(\rho) = +\infty$. 
In case $(2)$, $\lim_{\rho \to \rho_+} \dot{\lambda}_{H_r,d_r}(\rho) = +\infty$.
Further, $\lambda_{H_r,d_r}$ is convex in the interior of its domain. In particular, it is of class $C^2$ for $r=1$, and $\lim_{\rho \to \rho_0} \ddot{\lambda}_{H_r,d_r}(\rho) = +\infty$ for $r>1$ (Figure \ref{fig:4}).
\end{proposition}
\begin{figure}[htb]
\begin{center}
\begin{tikzpicture}[scale=1.3]
   \draw[->] (-1.2,0) -- (2,0);
  \draw[->] (0,-0.5) -- (0,1.5);
  \draw[thick] (0.3,0) .. controls (0.3,-1) and (1.2,-0.3) .. (1.4,0.7);
  \draw[thick, dashed] (1.4,0.7) .. controls (1.4,0.7) and (1.42,0.8) .. (1.43,0.9);
  \draw[dotted] (0.56,0) -- (0.56,-0.5);
  \node at (-1,1) {$H_r \leq \frac{n-r}{n}$};
  \node at (0.2,0.2) {$\rho_-$};
  \node at (0.6,0.2) {$\rho_0$};
  \node at (-0.15,-0.2) {$0$};
  \node at (1.15,1.2) {$\lambda_{H_r,d_r}$};
  \node at (2,-0.2) {$\rho$};
  \node at (-0.2,1.5) {$t$};
\end{tikzpicture}
\qquad 
\begin{tikzpicture}[scale=1.3]
   \draw[->] (-1.2,0) -- (2,0);
  \draw[->] (0,-0.5) -- (0,1.5);
  \draw[thick] (0.3,0) .. controls (0.3,-1) and (1.4,-0.3) .. (1.4,1);
  \draw[thick,blue] (0.2,0) .. controls (0.2,-1) and (1.5,-0.7) .. (1.5,0);
  \draw[thick,red] (0.1,0) .. controls (0.1,-1) and (1.6,-1) .. (1.6,-0.3);
  \draw[dotted] (0.57,0) -- (0.57,-0.5);
  \draw[dotted] (1.4,0) -- (1.4,1);
  \draw[dotted] (1.6,0) -- (1.6,-0.3);
  \draw[dotted] (0.75,0) -- (0.75,-0.65);
  \draw[dotted] (0.95,0) -- (0.95,-0.8);
  \node at (1.7,0.2) {$\rho_+$};
  \node at (-1,1) {$H_r > \frac{n-r}{n}$};
  \node at (0.2,0.2) {$\rho_-$};
  \node at (0.75,0.2) {$\rho_0$};
  \node at (-0.15,-0.2) {$0$};
  \node at (1.15,1.2) {$\lambda_{H_r,d_r}$};
  \node at (2,-0.2) {$\rho$};
  \node at (-0.2,1.5) {$t$};
\end{tikzpicture}
\end{center}
\caption{Behavior of $\lambda_{H_r,d_r}$ for $n>r$, $r$ odd, and $d_r < 0$.
For $H_1 > (n-1)/n$, $\lambda_{H_1,d_1}(\rho_+)$ is positive. 
When $r \geq 3$, $\lambda_{H_r,d_r}(\rho_+)$ may be positive (black curve), negative (red curve), or zero (blue curve) depending on the values of $H_r$ and $d_r$.}
\label{fig:4}
\end{figure}
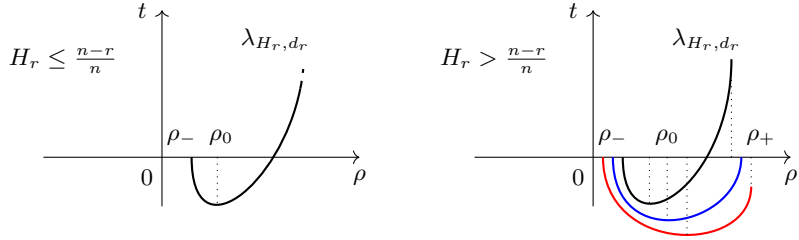
\begin{proof}
Our constraint for $\lambda_{H_r,d_r}$ to be well-defined is now 
\begin{equation}
\label{cond2}
-\sinh^{n-r}(\rho) < nH_r I_{n,r}(\rho)+d_r < \sinh^{n-r}(\rho), \qquad \rho > 0.
\end{equation}
We know that $nH_rI_{n,r}+d_r$ is an increasing function with $d_r < 0$ and $I_{n,r}(0) = 0$, so that 
$nH_rI_{n,r}(0)+d_r < 0$. The first inequality in \eqref{cond2} is then always satisfied for $\rho > \rho_- > 0$, where
$\rho_-$ is the unique solution of $nH_rI_{n,r}(\rho)+d_r+\sinh^{n-r}(\rho) = 0$. It is clear that $\rho_- \to 0$ if and only if $d_r \to 0$. 
The study of the second inequality goes along the lines of the corresponding one for $r$ even (Proposition~\ref{r-even-n>r1}).
Note that $\lim_{\rho \to \rho_-} \dot{\lambda}_{H_r,d_r}(\rho) = -\infty$ regardless of the value of $H_r$.
Also, $\lambda_{H_r,d_r}$ is decreasing on $(\rho_-,\rho_0)$, where $\rho_0$ is the only zero of $nH_rI_{n,r}+d_r$, then it increases beyond $\rho_0$.
Convergence at $\rho_-$ or $\rho_+$  and the statements involving the second derivative follow by \eqref{second-derivative} and similar arguments as in the proof of Proposition \ref{r-even-n>r1}.
We point out that for $r=1$ the term $(nH_rI_{n,r}(\rho)+d_r)^{(r-1)/r}$ equals $1$, so the second derivative of $\lambda_{H_r,d_r}$ is well-defined
over the interior of the whole domain. For $r > 1$ the same term vanishes at $\rho_0$, and this concludes the proof.
\end{proof}

Unlike the case when $r$ is even, the sign of $\lambda_{H_r,d_r}(\rho_+),$ for $H_r > (n-r)/n,$ $r>1$ odd, is not always positive. We discuss this point here below. Moreover we show that $\lambda_{H_1,d_1}(\rho_+)$ only takes positive values.

\begin{proposition}
\label{prop:special-cases}
The following statements hold.
\begin{enumerate}
\item \label{item1} If $H_1 > (n-1)/n$, then $\lambda_{H_1,d_1}(\rho_+) > 0$ for all $d_1 < 0$.
\item \label{item2} Let $2r-1>n>r\geq 3$, and $r$ odd. Then there exist values $H_r>(n-r)/n$ and $d_r<0$ such that 
$\lambda_{H_r,d_r}(\rho_+)$ is negative, positive, or zero.
\end{enumerate}
\end{proposition}
\begin{proof}
In  case (\ref{item1}),  it is well known that the rotational hypersurface generated by the curve defined by $\lambda_{H_1,d_1}$ is of class $C^2$. 
We show (\ref{item1}) by using Alexandrov reflection method with respect to vertical hyperplanes in $\mathbb H^n \times \mathbb R$.
Let $H_1 > (n-1)/n$ be fixed. Since the function defining $\lambda_{H_1,0}$ is non-negative and does not vanish, and $\lambda_{H_1,d_1}$ is continuous in $d_1$, then for $d_1 < 0$ close enough to $0$ we have $\lambda_{H_1,d_1}(\rho_+) > 0$.
Suppose there is a value of the parameter $d_1$ for which $\lambda_{H_1,d_1}(\rho_+)$ vanishes.
Consider the rotational hypersurface $S$ obtained after reflecting the graph of $\lambda_{H_1,d_1}$ across the $\rho$-axis, and then rotating about the $t$-axis.
Topologically $S$ is a product $S^1 \times S^{n-1}$, and is of class $C^2$.
Since $S$ is compact, we can take a vertical hyperplane $\Pi \subset \mathbb H^n \times \mathbb R$ corresponding to $\rho > 0$ large enough not intersecting $S$, and then move it towards $S$ until $\Pi \cap S \neq \varnothing$.
We keep moving $\Pi$ in the same way and reflect the portion of $S$ left behind $\Pi$ across $\Pi$. Since $\lambda_{H_1,d_1}(\rho_-) = 0$, there will be a first  intersection point between the reflected part of $S$ and $S$ itself.
The Maximum Principle then implies that $S$ has a symmetry with respect to a vertical hyperplane corresponding to some $\rho \in (\rho_-,\rho_+)$.
But this is a contradiction, as the hypersurface has rotational symmetry about $t = 0$. 
Continuity of $\lambda_{H_1,d_1}$ with respect to the parameters implies that there cannot be values of $d_1$ such that $\lambda_{H_1,d_1}(\rho_+)$ is negative.

As for (\ref{item2}), observe that for $H_r > (n-r)/n$ we have $\lambda_{H_r,0}(\rho_+)>0$, because the integrand function defining $\lambda_{H_r,0}$ is non-negative and does not vanish identically.
Continuity with respect to the parameter $d_r$ implies that $\lambda_{H_r,d_r}(\rho_+)>0$ for $d_r<0$ close enough to $0$.
We now show that $\lambda_{H_r,d_r}(\rho_+)<0$ for some $H_r > (n-r)/n$ and $d_r < 0$.
Let us introduce the function
$$g(\rho) \coloneqq \frac{nH_rI_{n,r}(\rho)+d_r}{\sinh^{n-r}(\rho)},$$
and note that we can rewrite $\lambda_{H_r,d_r}(\rho_+)$ as
$$\lambda_{H_r,d_r}(\rho_+)=\int_{\rho_-}^{\rho_+}\frac{g(\xi)^{\frac 1r}}{\sqrt{1-g(\xi)^{\frac 2r}}}\,d\xi.$$
We claim that, for any $d_r<0$ and $2r-n-1 > 0$, if $H_r$ is large enough then $g$ is convex on $(\rho_-,\rho_+)$.
So let $d_r < 0$ be fixed. By definition of $\rho_{\pm}$ we have
$$
H_r=\frac{|d_r|\pm\sinh^{n-r}(\rho_{\pm})}{nI_{n,r}(\rho_{\pm})}.
$$ 
Observe that $\rho_{\pm}\rightarrow 0$ if and only if $H_r\rightarrow\infty$ and $\rho_{\pm}\approx |d_r|^{\frac{1}{n}}H_r{}^{-\frac{1}{n}}$ as $H_r\rightarrow\infty$.
Therefore for any $\rho\in(\rho_-,\rho_+)$ we estimate
\begin{equation}\label{eq001}
\rho\approx\left(\frac{|d_r|}{H_r}\right)^{\frac 1n},\quad H_r\rightarrow\infty.
\end{equation}

Since $-\sinh^{n-r}(\rho)<nH_rI_{n,r}(\rho)+d_r<\sinh^{n-r}(\rho)$ holds on $(\rho_-,\rho_+)$, \eqref{eq001} and explicit computations give that for any $\rho\in(\rho_-,\rho_+)$ we have 
\begin{align*}
g''(\rho) & = nH_r\left(\frac{\sinh(\rho)}{\cosh(\rho)}\right)^{r-2}\left(\frac{r-1}{\cosh^2(\rho)}-(n-r)\right)\\
& \qquad +\frac{nH_rI_{n,r}(\rho)+d_r}{\sinh^{n-r+2}(\rho)}((n-r)\sinh^2(\rho)+n-r+1) \\
& > nH_r\left(\frac{\sinh(\rho)}{\cosh(\rho)}\right)^{r-2}\left(\frac{r-1}{\cosh^2(\rho)}-(n-r)\right)-\frac{(n-r)\sinh^2(\rho)+n-r+1}{\sinh^2(\rho)} \\
& \approx H_r^{\frac 2n}\left((2r-1-n)|d_r|^{\frac{r-2}{n}}H_r{}^{\frac{n-r}{n}}-(n-r+1)|d_r|^{-\frac{2}{n}}\right)-(n-r).
\end{align*}
When $H_r\rightarrow\infty$ the latter quantity diverges to $+\infty$ if $2r-1-n>0$, hence $g'' > 0$ on $(\rho_-,\rho_+)$.
Fix $H_r$ large enough such that $g$ is convex in $(\rho_-,\rho_+)$. 
Since $g(\rho_{\pm})=\pm 1$, then $g(\rho)<s(\rho)$ for any $\rho \in (\rho_-,\rho_+)$, where $s$ is the segment-line connecting $(\rho_-,-1)$ with $(\rho_+,1)$. 
Moreover the function $x \mapsto x^{1/r}/\sqrt{1-x^{2/r}}$ is increasing on $(-1,1)$. 
For such a choice of $H_r$ and $d_r$ we then have
$$
\lambda_{H_r,d_r}(\rho_+)<\int_{\rho_-}^{\rho_+}\frac{s(\xi)^\frac 1r}{\sqrt{1-s(\xi)^{\frac 2r}}}\,d\xi=\frac{\rho_+-\rho_-}{2}\int_{-1}^1\frac{u^\frac 1r}{\sqrt{1-u^{\frac 2r}}}\,du=0,
$$
as the latter integrand function is odd.

Continuity of $\lambda_{H_r,d_r}$ with respect to the parameters $H_r$ and $d_r$ implies the last assertion of  (\ref{item2}) at once.
\end{proof}

The proof of the next statement is left out, because the results can be seen by adapting the proof of Proposition \ref{r-even-n>r1} when $d_r = 0$.
\begin{proposition}
\label{r-odd-n>r2}
Assume $r$ odd, $n>r$, and $d_r=0$. 
\begin{enumerate}
\item If $0<H_r \leq (n-r)/n$, then $\lambda_{H_r,0}$ is defined on $[0,+\infty)$.
\item If $H_r > (n-r)/n$, then $\lambda_{H_r,0}$ is defined on $[0,\rho_+]$, where $\rho_+>0$ is the only solution of $\sinh^{n-r}(\rho)-nH_rI_{n,r}(\rho) = 0$. 
\end{enumerate}
Further, $\lambda_{H_r,0}$ is increasing and convex in the interior of its domain.
We have $\lambda_{H_r,0}(0) = 0 = \lim_{\rho \to 0} \dot{\lambda}_{H_r,0}(\rho)$. 
In case $(1)$, $\lambda_{H_r,0}$ is unbounded. 
In case $(2)$, $\lim_{\rho \to \rho_+} \dot{\lambda}_{H_r,0}(\rho) = +\infty$.
Finally, $\lim_{t \to 0} \ddot{\lambda}_{H_r,0}(\rho) = H_r{}^{1/r}$ (Figure \ref{fig:5}).
\end{proposition}
\begin{figure}[htb]
\begin{center}
\begin{tikzpicture}[scale=1.3]
  \draw[->] (-1.2,0) -- (2,0);
  \draw[->] (0,-0.5) -- (0,1.5);
  \draw[thick] (0,0) .. controls (1,0) and (1.3,0.5) .. (1.6,1);
  \draw[thick, dashed] (1.6,1) .. controls (1.7,1.2) and (1.7,1.2) .. (1.75,1.3);
  \node at (-1,1) {$H_r \leq \frac{n-r}{n}$};
  \node at (-0.15,-0.2) {$0$};
  \node at (2,-0.2) {$\rho$};
  \node at (1.2,1.2) {$\lambda_{H_r,0}$};
  \node at (-0.2,1.5) {$t$};
\end{tikzpicture}
\qquad 
\begin{tikzpicture}[scale=1.3]
   \draw[->] (-1.2,0) -- (2,0);
  \draw[->] (0,-0.5) -- (0,1.5);
  \draw[thick] (0,0) .. controls (0.5,0) and (1.5,0.01) .. (1.5,1);
  \draw[dotted] (1.5,1) -- (1.5,0);
  \node at (-1,1) {$H_r > \frac{n-r}{n}$};
  \node at (-0.15,-0.2) {$0$};
  \node at (1.5,-0.2) {$\rho_+$};
  \node at (1.15,1.2) {$\lambda_{H_r,0}$};
  \node at (2,-0.2) {$\rho$};
  \node at (-0.2,1.5) {$t$};
\end{tikzpicture}
\end{center}
\caption{Behavior of $\lambda_{H_r,0}$ for $n>r$ and $r$ odd.}
\label{fig:5}
\label{limits}
\end{figure}
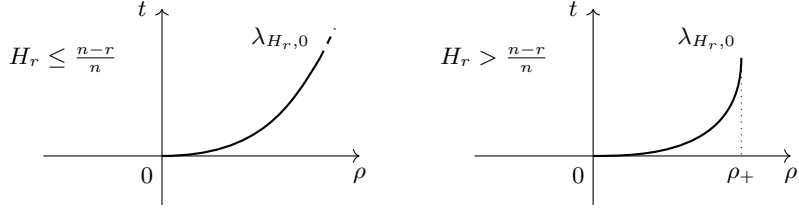

In order to prove the next result, one needs the analogue of formula \eqref{formula_integral} and Lemma~\ref{important_identity} for $r$ odd. 
We have $I_{2,1}(x) = \cosh(x) - 1$ and for $r \geq 3$ we compute
\begin{align*}
I_{r+1,r}(x) & = -\sinh(x)\biggl(\frac{1}{r-2}\tanh^{r-2}(x)+\frac{r-1}{(r-2)(r-4)}\tanh^{r-4}(x) \nonumber \\
& \qquad +\frac{(r-1)(r-3)}{(r-2)(r-4)(r-6)}\tanh^{r-6}(x)+\dots + \frac{(r-1)!!}{2(r-2)!!}\tanh(x)\biggr) \nonumber \\ 
& \qquad + \frac{(r-1)!!}{(r-2)!!}I_{2,1}(x).
\end{align*}
\begin{lemma}
Let $r \geq 3$ be an odd natural number. Then
\begin{align*}
\frac{(r-1)!!}{(r-2)!!} & = 1+\frac{1}{r-2}+\frac{r-1}{(r-2)(r-4)}+\frac{(r-1)(r-3)}{(r-2)(r-4)(r-6)} +\\
& \qquad + \dots + \frac{(r-1)(r-3)\cdots 4}{(r-2)(r-4)\cdots 3},
\end{align*}
where, for all $r,$ the sum on the right-hand side must be truncated in such a way that all summands are positive.
\end{lemma}
The next two results can be proved  following the proof of Propositions \ref{r_even-2} and \ref{r-even-n=r}.

\begin{proposition}
Assume $r$ odd, $n>r$, and $d_r > 0$.
\begin{enumerate}
\item If $0<H_r < (n-r)/n$, then $\lambda_{H_r,d_r}$ is defined on $[\rho_-,+\infty)$, where $\rho_- > 0$
is the only solution of $\sinh^{n-r}(\rho)-(nH_rI_{n,r}(\rho)+d_r) = 0$. 
\item If $H_r = (n-r)/n$, then when $n=r+1$ we need $d_1 < 1$ or $d_r < (r-1)!!/(r-2)!!$ for $r>1,$ in order for $\lambda_{H_r,d_r}$ to be well-defined, whereas for $n>r+1$ we have no constraint.
Under such conditions, the results in the previous point hold.
\item If $H_r > (n-r)/n$, set $\tau > 0$ such that $\tanh^r(\tau) = (n-r)/nH_r$. Then $d_r < \sinh^{n-r}(\tau)-nH_rI_{n,r}(\tau)$,
for $\lambda_{H_r,d_r}$ to be defined. So $\lambda_{H_r,d_r}$ is a function on $[\rho_-,\rho_+] \subset (0,+\infty)$, where $\sinh^{n-r}(\rho_{\pm})-(nH_rI_{n,r}(\rho_{\pm})+d_r) = 0$. 
\end{enumerate}
Further, $\lambda_{H_r,d_r}$ is increasing in the interior of its domain.
In cases $(1)$--$(2)$, $\lambda_{H_r,d_r}(\rho_-) = 0$, $\lim_{\rho \to \rho_-} \dot{\lambda}_{H_r,d_r}(\rho) = +\infty$, $\lambda_{H_r,d_r}$ is unbounded, and is concave in the interior of its domain.
In case $(3)$, $\lambda_{H_r,d_r}(\rho_-) = 0$, $\lim_{\rho \to \rho_{\pm}}\dot{\lambda}_{H_r,d_r}(\rho) = +\infty$, $\lambda_{H_r,d_r}$ has a unique inflection point in $(\rho_-,\rho_+)$, and goes from being concave to convex (Figure \ref{fig:6}).
\end{proposition}
\begin{figure}[htb]
\begin{center}
\begin{tikzpicture}[scale=1.3]
  \draw[->] (-1.2,0) -- (2,0);
  \draw[->] (0,-0.5) -- (0,1.5);
  \draw[thick] (0.3,0) .. controls (0.3,0.5) and (1,1) .. (1.5,1.3);
  \draw[thick,dashed] (1.5,1.3) .. controls (1.7,1.42) and (1.7,1.41) .. (1.75,1.44);
  \node at (-1,1) {$H_r \leq \frac{n-r}{n}$};
  \node at (0.3,-0.2) {$\rho_-$};
  \node at (-0.15,-0.2) {$0$};
  \node at (2,-0.2) {$\rho$};
  \node at (2,1.2) {$\lambda_{H_r,d_r}$};
  \node at (-0.2,1.5) {$t$};
\end{tikzpicture}
\qquad 
\begin{tikzpicture}[scale=1.3]
   \draw[->] (-1.2,0) -- (2,0);
  \draw[->] (0,-0.5) -- (0,1.5);
  \draw[thick] (0.3,0) .. controls (0.3,0.6) and (1.5,0.5) .. (1.5,1);
  \draw[dotted] (1.5,1) -- (1.5,0);
  \node at (-1,1) {$H_r > \frac{n-r}{n}$};
  \node at (0.3,-0.2) {$\rho_-$};
  \node at (-0.15,-0.2) {$0$};
  \node at (1.5,-0.2) {$\rho_+$};
  \node at (1.15,1.2) {$\lambda_{H_r,d_r}$};
  \node at (2,-0.2) {$\rho$};
  \node at (-0.2,1.5) {$t$};
\end{tikzpicture}
\end{center}
\caption{Behavior of $\lambda_{H_r,d_r}$ for $n>r$, $r$ odd, and $d_r > 0$.}
\label{fig:6}
\end{figure}
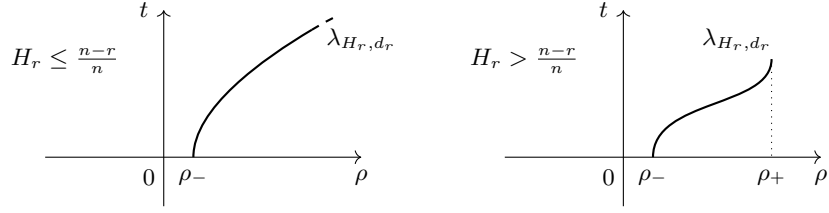

\begin{proposition}
\label{r-odd-n=r}
Assume $n=r$ odd. Then $\lambda_{H_n,d_n}$ is well-defined for $d_n < 1$. Set $I_n \coloneqq I_{n,n}$.
\begin{enumerate}
\item If $d_n < -1$, then $\lambda_{H_n,d_n}$ is defined on $[\rho_-,\rho_+]$,
where $\rho_-$ is the only solution of $nH_nI_n(\rho)+d_n=-1$, and $\rho_+$ is the only solution of $nH_nI_n(\rho)+d_n=1$.
\item If $-1\leq d_n < 1$, then $\lambda_{H_n,d_n}$ is defined on $[0,\rho_+]$, where $\rho_+$
is defined as above. 
\end{enumerate}
Further, $\lambda_{H_n,d_n}$ is convex in the interior of its domain. 
Set $\rho_0$ to be the only solution of $nH_nI_n(\rho)+d_n = 0$.
In case $(1)$, we have $\lambda_{H_n,d_n}(\rho_-) = 0$, $\dot{\lambda}_{H_n,d_n}(\rho) < 0$ for $\rho_- < \rho < \rho_0$, $\dot{\lambda}_{H_n,d_n}(\rho) > 0$ for $\rho > \rho_0$, $\lambda_{H_n,d_n}(\rho_+) < 0$, and $\lim_{\rho \to \rho_+}\dot{\lambda}_{H_n,d_n}(\rho) = +\infty$.
In case $(2)$, one finds $\dot{\lambda}_{H_n,d_n}(0) = d_n^{1/n}/(1-d_n^{2/n})^{1/2}$, and $\lim_{d_n \to -1} \dot{\lambda}_{H_n,d_n}(0) = -\infty$.
For $d_n < 0$ the function $\lambda_{H_n,d_n}$ first decreases then increases, and the sign of $\lambda_{H_n,d_n}$ depends on the value of $d_n$, whereas for $d_n \geq 0$ the function $\lambda_{H_n,d_n}$ is increasing on the whole domain. 
Moreover, $\lim_{\rho \to 0} \ddot{\lambda}_{H_n,0}(\rho) = H_n{}^{1/n}$, and $\lim_{\rho \to \rho_0} \ddot{\lambda}_{H_n,d_n}(\rho) = +\infty$ (Figure~\ref{fig:7}).
\end{proposition}
\begin{figure}[htb]
\begin{center}
\begin{tikzpicture}[scale=1.3]
   \draw[->] (-1.2,0.5) -- (2,0.5);
  \draw[->] (0,-0.5) -- (0,1.5);
  \draw[thick] (0.3,0.5) .. controls (0.3,-1) and (1.4,-0.7) .. (1.4,0.3);
  \draw[dotted] (1.4,0.5) -- (1.4,0.3);
  \draw[dotted] (0.85,0.5) -- (0.85,-0.5);
  \node at (-0.7,1) {$d_n < -1$};
  \node at (0.3,0.7) {$\rho_-$};
  \node at (-0.15,-0.2) {$0$};
  \node at (1.5,0.7) {$\rho_+$};
  \node at (1.15,1.2) {$\lambda_{H_n,d_n}$};
  \node at (2,0.7) {$\rho$};
  \node at (0.9,0.7) {$\rho_0$};
  \node at (-0.2,1.5) {$t$};
\end{tikzpicture}
\qquad 
\begin{tikzpicture}[scale=1.3]
   \draw[->] (-1.2,0.5) -- (2,0.5);
  \draw[->] (0,-0.5) -- (0,1.5);
  \draw[thick,red] (0,0.5) .. controls (0.1,-0.1) and (1,-0.5) .. (1,0.9);
  \draw[thick,orange] (0,0.5) .. controls (0.5,0.8) and (0.6,1) .. (0.6,1.5);
  \draw[thick,violet] (0,0.5) .. controls (0.5,0.5) and (0.8,0.5) .. (0.8,1.2);
  \draw[thick,blue] (0,0.5) .. controls (0.2,-0.7) and (1.2,-0.5) .. (1.2,0.5);
  \draw[thick,black] (0,0.5) .. controls (0.1,-1) and (1.4,-0.7) .. (1.4,0);
  \draw[dotted] (1.4,0) -- (1.4,0.5);
  \draw[dotted] (1,0.5) -- (1,0.9);
  \draw[dotted] (0.8,1.2) -- (0.8,0.5);
  \draw[dotted] (0.6,0.5) -- (0.6,1.5);
  \node at (-1,1) {$-1 \leq d_n < 1$};
  \node at (-0.15,-0.2) {$0$};
  \node at (1.6,0.7) {$\rho_+$};
  \node at (1.5,1.2) {$\lambda_{H_n,d_n}$};
  \node at (2,0.7) {$\rho$};
  \node at (-0.2,1.5) {$t$};
\end{tikzpicture}
\end{center}
\caption{Behavior of $\lambda_{H_n,d_n}$ for $n$ odd and $H_n > 0$. 
When $-1 \leq d_n < 1$, we distinguish four cases, i.e.\ $d_n = -1$ (black), $-1 < d_n < 0$ (red), $d_n = 0$ (violet), $0 < d_n < 1$ (orange).
The blue curve corresponds to a value of $d_n \in (-1,0)$ for which $\lambda_{H_n,d_n}(\rho_+)$ vanishes.}
\label{fig:7}
\end{figure}
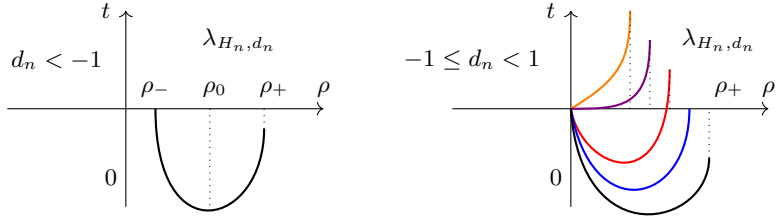

\begin{proof}
The only part of the proof differing  from the proof of Proposition \ref{r-even-n=r} is  about the sign  of 
$\lambda_{H_n,d_n}(\rho_+)$.
We look first at the case $d_n < -1$. Since $nH_nI_n+d_n$ is convex, $nH_nI_n(\rho)+d_n < s(\rho)$,
where $s$ is the line passing through the points $(\rho_-,-1)$ and $(\rho_+,1)$.
Now $nH_nI_n(\rho) + d_n < s(\rho)$ for $\rho \in (\rho_-,\rho_+)$, so we also have
$$\frac{(nH_nI_n(\rho)+d_n)^{\frac1n}}{\sqrt{1-(nH_nI_n(\rho)+d_n)^{\frac2n}}} < \frac{s(\rho)^{\frac1n}}{\sqrt{1-s(\rho)^{\frac2n}}},$$
as the function $x \mapsto x^{1/n}/\sqrt{1-x^{2/n}}$ is increasing.
But the integral of the latter quantity over $(\rho_-,\rho_+)$ is computed to be zero, as the integrand function is odd:
\begin{align*}
\int_{\rho_-}^{\rho_+} \frac{s(\xi)^{\frac1n}}{\sqrt{1-s(\xi)^{\frac2n}}}\,d\xi = \frac{\rho_+-\rho_-}{2}\int_{-1}^1 \frac{u^{\frac1n}}{\sqrt{1-u^{\frac2n}}}\,du = 0.
\end{align*}
This shows $\lambda_{H_n,d_n}(\rho_+) < 0$. The same holds when $d_n = -1$, the only difference being that $\rho_- = 0$.
Since $\lambda_{H_n,d_n}(\rho_+)$ depends continuously on $d_n$, and for $d_n \geq 0$ we have $\lambda_{H_n,d_n}(\rho_+) > 0$, 
there must be a $d_n \in (-1,0)$ such that $\lambda_{H_n,d_n}(\rho_+) = 0$.
\end{proof}

As in the case of $r$ even (cf.\ Proposition \ref{prop:c2-reg}), before proceeding with the classification result, we study the regularity of the $H_r$-hypersurface generated by a rotation of the graph of $\lambda_{H_r,d_r}.$ 

\begin{proposition}
\label{prop:c2reg-disp}
Let $n\geq r$, $r$ odd. Then the hypersurface generated by the graph of $\lambda_{H_r,d_r}$ is of class $C^2$ at $\rho = \rho_+$, when the latter exists, and it is of class $C^2$ at $\rho=\rho_-$ if and only if $n>r$, or $n=r$ and $d_n\in(-\infty,-1)\cup\{0\}$. When $n=r$ and $d_r\in[-1,0)\cup(0,1)$, there is a conical singularity at $\rho=0$. Moreover, if $r\geq 3$ and $d_r \neq 0$, the hypersurface is $C^2$-singular at any critical point of the function $\lambda_{H_r,d_r}$.

\end{proposition}
\begin{proof}
The first part of the proof is an application of the same argument as in Proposition \ref{prop:c2-reg}. If $r=1$ it is well known that the corresponding hypersurface is smooth. Now let $r\geq 3$ and let $\rho_0$ be a critical point of $\lambda_{H_r,d_r}$. By \eqref{lambda} we have that $nH_rI_{n,r}(\rho_0)+d_r=0$. By \eqref{second-derivative} it follows that $\lim_{\rho\rightarrow\rho_0}\ddot\lambda_{H_r,d_r}(\rho)=+\infty$. Using \eqref{data} we can see that $\lim_{\rho\rightarrow\rho_0}k_n(\rho)=+\infty$, hence $|A|^2$ blows up near $\rho_0$.
\end{proof}

\begin{remark} The same kind of singularities appears in the construction of rotationally invariant higher order translators, i.e.\ hypersurfaces with $H_r=g(\nu,\partial/\partial t)$, where $r>1$ and $\nu$ is the unit normal, see \cite{lima-pipoli}.
\end{remark}
 
We now proceed with the classification results. We recover results by Bérard--Sa Earp \cite{berard-earp}, Elbert--Sa Earp \cite[Section 6]{elb-earp} and de Lima--Manfio--dos Santos \cite[Theorem 1 and 2]{lima}. 
Recall that a slice is any subspace $\mathbb H^n\times\{t\} \subset \mathbb H^n \times \mathbb R$, and its origin was defined as its intersection with the $t$-axis.

\begin{theorem}
\label{str-thm5}
Assume $r$ odd, $n>r$, and $d_r < 0$. 
By reflecting the rotational hypersurface given by the graph of $ \lambda_{H_r,d_r}$ across suitable slices, we get an immersed $H_r$-hypersurface. 
\begin{enumerate}
\item If $0<H_r \leq (n-r)/n$, the hypersurface generated by the graph of $\lambda_{H_r,d_r}$, together with its reflection across the slice
$\mathbb H^n \times \{0\}$, is an annulus with self-intersections along a sphere centered at the origin of the slice $\mathbb H^n \times \{0\}$. 
The hypersurface is of class $C^2$ for $r=1$, and of class $C^1$ for $r \geq 3$. 
In the latter case, the singular set consists of two spheres of radius $\rho_0$ contained in the slices $\mathbb H^n \times \{\pm \lambda_{H_r,d_r}(\rho_0)\}$ and centered at their origin.
\item If $H_r > (n-r)/n$, then we distinguish two subcases. If $r=1$, the hypersurface generated by the graph of $\lambda_{H_r,d_r}$,
together with its reflection across the slice $\mathbb H^n \times \{0\}$ and vertical translations of integral multiples of 
$2\lambda_{H_r,d_r}(\rho_+)$, is a $C^2$ nodoid  with self-intersections along infinitely many spheres centered at the origin of distinct slices. 
If $r \geq 3$, we have two possibilities. First, one may get nodoids as in the $r=1$ case, except that they are not $C^2$-regular (singularities appear along infinitely many spheres of radius $\rho_0$ in distinct slices).
Second, one may get compact hypersurfaces with the topology of $S^1 \times S^{n-1}$ with $C^2$ singularities along two spheres of radius $\rho_0$ contained in the slices $\mathbb H^n \times \{\pm \lambda_{H_r,d_r}(\rho_0)\}$ and centered at their origin.
\end{enumerate}
\end{theorem}

\begin{theorem}
\label{str-thm6}
Assume $r$ odd, $n>r$, and $d_r = 0$. 
Then the rotational hypersurface given by the graph of $\lambda_{H_r,0}$ is a complete embedded $H_r$-hypersurface, possibly after reflection across a suitable slice.
\begin{enumerate}
\item If $0<H_r \leq (n-r)/n$, the hypersurface generated by the graph of $\lambda_{H_r,0}$ is  an entire graph of class $C^2$ tangent to the slice $\mathbb H^n \times \{0\}$ at the origin.
\item If $H_r > (n-r)/n$,  the hypersurface generated by the graph of $\lambda_{H_r,0}$, together with its reflection across the slice $\mathbb H^n \times \{\lambda_{H_r,0}(\rho_+)\}$, is a class $C^2$ sphere.
\end{enumerate}
\end{theorem}

\begin{theorem}
\label{str-thm7}
Assume $r$ odd, $n>r$, and $d_r > 0$. 
By reflecting the rotational hypersurface given by the graph of $\lambda_{H_r,d_r}$ across suitable  slices, we get a complete non-compact embedded $H_r$-hypersurface.
\begin{enumerate}
\item If $0<H_r \leq (n-r)/n$, the hypersurface generated by the graph of $\lambda_{H_r,d_r}$, together with its reflection across the slice $\mathbb H^n \times \{0\}$, is a class $C^2$ annulus.
When $n=r+1$ and $H_r=1/(r+1)$ the same holds, provided that $d_r$ is smaller than $(r-1)!!/(r-2)!!$ for $r>1$, or smaller than $1$ for $r=1$.
\item If $H_r > (n-r)/n$, the hypersurface generated by the graph of $\lambda_{H_r,d_r}$ together with its reflections  across the slices $\mathbb H^n \times \{k\lambda_{H_r,d_r}(\rho_+)\}$, $k \in \mathbb Z$, is a class $C^2$ onduloid.
\end{enumerate}
\end{theorem}

\begin{theorem}
\label{str-thm8}
Assume $n=r$ odd and $H_n > 0$. Then the hypersurface generated by the graph of $\lambda_{H_n,d_n}$, together with its reflection across the slice $\mathbb H^n \times \{\lambda_{H_n,d_n}(\rho_+)\}$ is a class $C^2$ sphere  if  $d_n = 0$, and a peaked sphere if $0 < d_n < 1$.

When $-1 \leq d_n < 0$, the hypersurface generated by the graph of $\lambda_{H_n,d_n}$, together with its reflections across suitable slices, has self-intersections and we have three possibilities: it may be a generalized horn torus, a portion of generalized spindle torus, or a nodoid. In all cases, the hypersurface has $C^2$ singularities, cf.\ Table \ref{table:3}.

When $d_n < -1$, the hypersurface generated by the graph of $\lambda_{H_n,d_n}$, together with its reflection across 
the slice $\mathbb H^n \times \{0\}$ and vertical translations of integral multiples of 
$2\lambda_{H_n,d_n}(\rho_+)$, is an immersed nodoid with $C^2$ singularities along infinitely many spheres of radius $\rho_0$ in distinct slices and centered at their origin.
\end{theorem}

Tables \ref{table:1}--\ref{table:3} summarize the above results. 
We describe the shape of the hypersurfaces and specify their homeomorphism type when the topology is easily described.

\begin{table}[ht]
  \centering
  \caption{Rotation $H_r$-hypersurfaces in $\mathbb H^n \times \mathbb R$ with $H_r > (n-r)/n$}
 \begin{tabular}{p{25mm}p{30mm}p{40mm}M{10mm}} \toprule
        \textbf{Parameters} & \textbf{Shape / Topology} & \textbf{Singularities} & \textbf{Figure} \\ \midrule
        $d_r>0$  & onduloid / $S^{n-1} \times \mathbb R$ & \ding{55} & \ref{fig:2}, \ref{fig:6} \\ \cmidrule{1-4}
        $d_r=0$  & sphere / $S^n$ & \ding{55} & \ref{fig:1}, \ref{fig:5} \\ \cmidrule{1-4}
        $d_r<0$, $r$ even & singular onduloid / $S^{n-1} \times \mathbb R$ & infinitely many copies of $S^{n-1}$ given by cusps in horizontal slices & \ref{fig:1}\\ \cmidrule{1-4}
        \multirow{2}{*}[-2.5em]{$d_r<0$, $r$ odd}  & nodoid  & $|A|^2\rightarrow\infty$ on infinitely many copies of $S^{n-1}$ in horizontal slices if $r \geq 3$, else smooth & \ref{fig:4} \\ \cmidrule(l){2-4}
         & $S^{n-1} \times S^1$ & $|A|^2\rightarrow\infty$ on two copies of $S^{n-1}$ in horizontal slices & \ref{fig:4} \\ \bottomrule
    \end{tabular}
\label{table:1}
\end{table}
\begin{remark}
Let us comment on the last case of Table \ref{table:1}, i.e.\ $d_r<0$ and $r$ odd. 
Both types of hypersurfaces noted there occur depending on the value of $H_r$ and $d_r$. Also, $S^{n-1} \times S^1$ occurs only if $r \geq 3$.
As $\lambda_{H_r,d_r}(\rho_+) \neq 0$ gets closer to the $\rho$-axis, the corresponding nodoids get more self-intersections, and the topology of the hypersurface becomes non-trivial.
\end{remark}

\begin{table}[ht]
  \centering
  \caption{Rotation $H_r$-hypersurfaces in $\mathbb H^n \times \mathbb R$ with $0<H_r \leq (n-r)/n$}
    \begin{tabular}{p{25mm}p{30mm}p{40mm}M{10mm}} \toprule
        \textbf{Parameters}   & \textbf{Shape / Topology} & \textbf{Singularities} & \textbf{Figure} \\ \midrule
        $d_r>0$  & unbounded annulus / $S^{n-1} \times \mathbb R$ & \ding{55} & \ref{fig:2}, \ref{fig:6} \\ \cmidrule{1-4}
        $d_r=0$  & entire graph / $\mathbb R^n$ & \ding{55} & \ref{fig:1}, \ref{fig:5} \\ \cmidrule{1-4}
        $d_r<0$, $r$ even & singular annulus / $S^{n-1} \times \mathbb R$ & a copy of $S^{n-1}$ given by cusps in the slice $t=0$ & \ref{fig:1}\\ \cmidrule{1-4}
        $d_r<0$, $r$ odd & singular annulus with self-intersections along a copy of $S^{n-1}$ in $\mathbb H^n \times \{0\}$ & $|A|^2\rightarrow\infty$ on two copies of $S^{n-1}$ in horizontal slices if $r \geq 3$, else smooth & \ref{fig:4} \\ \bottomrule
   \end{tabular}
\label{table:2}
\end{table}

\begin{table}[ht]
  \centering
  \caption{Rotation $H_n$-hypersurfaces in $\mathbb H^n\times\mathbb R$ with $H_n>0$}
    \begin{tabular}{p{25mm}p{30mm}p{40mm}M{10mm}} \toprule
        \textbf{Parameters}   & \textbf{Shape / Topology} & \textbf{Singularities} & \textbf{Figure} \\ \midrule
        $d_n < -1$, $n$ odd & nodoid & $|A|^2\rightarrow\infty$ at infinitely many copies of $S^{n-1}$ in horizontal slices & \ref{fig:7} \\ \cmidrule{1-4}
         & nodoid & $|A|^2 \to \infty$ at infinitely many points on the $t$-axis and copies of $S^{n-1}$ in horizontal slices & \ref{fig:7} \\ \cmidrule(l){2-4}
        \multirow{2}{*}{\parbox{3cm}{$-1 \leq d_n < 0$ \\ $n$ odd}} & generalized horn torus & $|A|^2 \to \infty$ at two copies of $S^{n-1}$ in horizontal slices and at one point on the $t$-axis & \ref{fig:7}\\ \cmidrule(l){2-4}
        								      & portion of generalized spindle torus & $|A|^2 \to \infty$ at two copies of $S^{n-1}$ in horizontal slices and at two points on the $t$-axis & \ref{fig:7} \\ \cmidrule{1-4}
        $d_n < 0$, $n$ even & singular onduloid / $S^{n-1} \times \mathbb R$ & infinitely many copies of $S^{n-1}$ given by cusps in horizontal slices & \ref{fig:3} \\ \cmidrule{1-4}
        $d_n = 0$  & sphere / $S^n$ & \ding{55} & \ref{fig:3}, \ref{fig:7} \\ \cmidrule{1-4}
        $0 < d_n < 1$ & peaked sphere / $S^n$ & $|A|^2\rightarrow\infty$ at two points on the $t$-axis & \ref{fig:3}, \ref{fig:7} \\ \bottomrule
    \end{tabular}
\label{table:3}
\end{table}
\begin{remark}
When $-1\leq d_n<0$ and $n$ is odd, all three cases in Table \ref{table:3} occur depending on the value of the parameter $d_n$.
\end{remark}

\section{Translation \texorpdfstring{$H_r$}{Hr}-hypersurfaces}
\label{translation-surfaces}

In the proof of Theorem \ref{main-thm}, besides rotation hypersurfaces, we will need further $H_r$-hypersurfaces as barriers. The suitable ones are invariant under hyperbolic translation in $\mathbb H^n\times \mathbb R$ with $r$-th mean curvature $H_r>(n-r)/n.$ Hyperbolic translations in $\mathbb H^n\times \mathbb R$ are hyperbolic translations in a slice $\mathbb H^n \times \{t\}$ extended to be constant on the vertical component, and will be described precisely later.
When $0<H_r< (n-r)/n$, smooth complete hypersurfaces invariant under hyperbolic translation are treated in \cite{lima}.
The case $r=1$ has already been studied in \cite{berard-earp}, and an explicit description for $n=2$ has been given by Manzano \cite{manzano}. 
Therefore, we restrict to the case  $r>1.$

As in Section \ref{rotational-surfaces}, given $n$, $r$, and $H_r > 0$, one finds a one-parameter family of functions describing the profile of such translation hypersurfaces.  
Since we do not aim to give a complete classification of translation hypersurfaces, we will choose the parameter to be zero (see \eqref{def:mu} below), and we will only describe a portion of the hypersurface. 
This will be enough for our purposes.

Let us recall the construction of translation hypersurfaces in $\mathbb H^n \times \mathbb R$ by B\'erard--Sa Earp \cite{berard-earp}. 
For simplifying the notation, we denote the zero-section $\mathbb H^n \times \{0\}$  by $\mathbb H^n.$
Take $\gamma$ in $\mathbb H^n$ to be a geodesic passing through $0$. 
We define $V$ to be the vertical plane $\{(\gamma(\rho),t): t,\rho \in \mathbb R\}$. 
We now take $\pi$ to be a totally geodesic hyperplane in $\mathbb H^n$ orthogonal to $\gamma$ at the origin.
We consider hyperbolic translations along  a geodesic $\delta$  passing through $0$ in $\pi,$  repeated slice-wise to get isometries   of $\mathbb H^n \times \mathbb R$.
Now take a curve defined by $c(\rho) \coloneqq (\tanh(\rho/2), \mu(\rho))$ in $V$, where $\mu$ is to be determined.
For any $\rho > 0$, consider the section $\mathbb H^n \times \{\mu(\rho)\}$, and move the point $c(\rho)$ via the above hyperbolic translations.
On each slice, this gives a hypersurface $M_{\rho}$ in $\mathbb H^n \times \{\mu(\rho)\}$ through $c(\rho)$. 
Hence the curve defined by $c$ generates a translation hypersurface $M = \cup_{\rho} M_{\rho}$ in  $\mathbb H^n \times \mathbb R$.

The principal curvatures of the hypersurface $M,$ with respect to the unit normal pointing upwards are
\begin{equation*}
k_1 = \dots = k_{n-1} = \frac{\dot{\mu}}{(1+\dot{\mu}^2)^{\frac12}}\tanh(\rho), \qquad k_n = \frac{\ddot{\mu}}{(1+\dot{\mu}^2)^{\frac32}}.
\end{equation*}
Then
$$nH_r = (n-r)\tanh^r(\rho)\frac{\dot{\mu}^r}{(1+\dot{\mu}^2)^{r/2}}+\tanh^{r-1}(\rho)\frac{r\dot{\mu}^{r-1}\ddot{\mu}}{(1+\dot{\mu}^2)^{\frac{r+2}{2}}}.$$
This is equivalent to the identity
\begin{equation}
\label{start2}
nH_r\frac{\cosh^{n-1}(\rho)}{\sinh^{r-1}(\rho)} = \frac{d}{d \rho} \left(\cosh^{n-r}(\rho)\frac{\dot{\mu}^r}{(1+\dot{\mu}^2)^{\frac{r}{2}}}\right), \qquad r=1,\dots,n.
\end{equation}
Note that now the integrals 
$$\int_0^{\rho} \frac{\cosh^{n-1}(\tau)}{\sinh^{r-1}(\tau)}\,d\tau$$
are not well-defined for $r > 1$, because 
$$\int_0^{\rho} \frac{\cosh^{n-1}(\tau)}{\sinh^{r-1}(\tau)}\,d\tau \geq \int_0^{\rho} \frac{\cosh^{r-1}(\tau)}{\sinh^{r-1}(\tau)}\,d\tau \geq \int_0^{\rho} \cotgh(\tau)\,d\tau = \infty.$$
We then choose $\epsilon > 0$ and define
$$J_{n,r,\epsilon}(\rho) \coloneqq \int_{\epsilon}^{\rho} \frac{\cosh^{n-1}(\tau)}{\sinh^{r-1}(\tau)}\,d\tau, \quad \text{ and } \quad J_{n,1}(\rho) \coloneqq \int_0^{\rho} \cosh^{n-1}(\tau)\,d\tau, \qquad \rho > 0.$$
Then one can integrate \eqref{start2} twice and set the constant of integration to be zero, so as to obtain
\begin{equation}
\label{def:mu}
\mu_{H_r,\epsilon}(\rho) = \int_{\epsilon}^{\rho} \frac{(nH_rJ_{n,r,\epsilon}(\xi))^{\frac{1}{r}}}{\sqrt{\cosh^{\frac{2(n-r)}{r}}(\xi)-(nH_rJ_{n,r,\epsilon}(\xi))^{\frac{2}{r}}}}\,d\xi, \qquad \rho \geq \epsilon.
\end{equation}
Again, $\mu$ depends on $H_r$, and $\epsilon$, so we write $\mu_{H_r,\epsilon}$ to be precise.

\begin{remark}
\label{signs2}
Note that we have defined $\mu_{H_r,\epsilon}$ in \eqref{def:mu} for $\rho \geq \epsilon$. 
This is because we are only interested in the portion of translation hypersurface described by the graph of $\mu_{H_r,\epsilon}$ for $\rho \geq \epsilon$.
The tangent line to the curve described by $\mu_{H_r,\epsilon}$ at $\rho = \epsilon$ is horizontal for all $r$, and $\mu_{H_r,\epsilon}$ is increasing for $\rho > \epsilon$.
The second derivative of $\mu_{H_r,\epsilon}(t)$ is computed as
\begin{equation}
\label{second-derivative-mu}
\ddot{\mu}_{H_r,\epsilon}(\rho) = \frac{\sinh(\rho)\cosh^{\frac{2(n-r)}{r}-1}(\rho)\left(nH_r\frac{\cosh^n(\rho)}{\sinh^r(\rho)}-(n-r)(nH_rJ_{n,r,\epsilon}(\rho))\right)}{r(nH_rJ_{n,r,\epsilon}(\rho))^{\frac{r-1}{r}}\left(\cosh^{\frac{2(n-r)}{r}}(\rho)-(nH_rJ_{n,r,\epsilon}(\rho))^{\frac{2}{r}}\right)^{\frac{3}{2}}}.
\end{equation}
This expression will be used when studying the convexity of  $\mu_{H_r,\epsilon}$  and its regularity  up to second order.
\end{remark}

\begin{remark}
\label{concavity-j}
Let us discuss a few details on $J_{n,r,\epsilon}$ for $r>1$ and $\rho>\epsilon.$
It is clear that $J_{n,r,\epsilon}(\epsilon) = 0$ and $\lim_{\rho \to +\infty} J_{n,r,\epsilon}(\rho) = +\infty.$ Further, $J_{n,r,\epsilon}'(\rho) > 0$ for $\rho\geq \epsilon.$ Hence $J_{n,r,\epsilon}$ is a bijection between $(\epsilon,\infty)$ and $(0,+\infty)$.

For $n>r$, we have the asymptotic behavior $(n-r)J_{n,r,\epsilon}(\rho) \approx \cosh^{n-r}(\rho)$ for $\rho \to+\infty$, and for $n = r$ we have $J_{n,n,\epsilon}(\rho) \approx \rho$ for $\rho \to +\infty$.
\end{remark}

We fix $r > 1$  $H_r > (n-r)/n$, and $\epsilon>0$, and study the function
$$\mu_{H_r,\epsilon}(\rho) \coloneqq \int_{\epsilon}^{\rho} \frac{(nH_rJ_{n,r,\epsilon}(\xi))^{\frac{1}{r}}}{\sqrt{\cosh^{\frac{2(n-r)}{r}}(\xi)-(nH_rJ_{n,r,\epsilon}(\xi))^{\frac{2}{r}}}}\,d\xi.$$

\begin{proposition}
\label{curve-translation-prop}
Let $r > 1,$ $H_r > (n-r)/n$,  and fix $\epsilon>0.$ 
Then $\mu_{H_r,\epsilon}$ is defined on $[\epsilon,\rho_+^{\epsilon}]$, where 
$\rho_+^{\epsilon}$ is the only solution of $\cosh^{n-r}(\rho)-nH_rJ_{n,r,\epsilon}(\rho) = 0$.
We have $\mu_{H_r,\epsilon}(\epsilon) = 0 = \dot{\mu}_{H_r,\epsilon}(\epsilon)$, $\dot{\mu}_{H_r,\epsilon}(\rho) > 0$ for $\rho \in (\epsilon,\rho_+^{\epsilon})$, $\lim_{\rho \to \rho_+^{\epsilon}} \dot{\mu}_{H_r,\epsilon}(\rho) = +\infty$,
and $\mu_{H_r,\epsilon}$ is convex in the interior of its domain. Further, $\lim_{\rho \to \epsilon} \ddot{\mu}_{H_r,\epsilon}(\rho) = +\infty$ (Figure \ref{fig:9}).

\end{proposition}
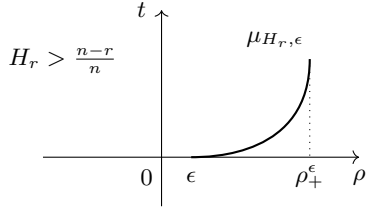
\begin{figure}[htb]
\begin{center}
\begin{tikzpicture}[scale=1.3]
   \draw[->] (-1.2,0) -- (2,0);
  \draw[->] (0,-0.5) -- (0,1.5);
  \draw[thick] (0.3,0) .. controls (0.5,0) and (1.5,0.01) .. (1.5,1);
  \draw[dotted] (1.5,1) -- (1.5,0);
  \node at (-1,1) {$H_r > \frac{n-r}{n}$};
  \node at (0.3,-0.2) {$\epsilon$};
  \node at (-0.15,-0.2) {$0$};
  \node at (1.5,-0.2) {$\rho_+^{\epsilon}$};
  \node at (1.15,1.2) {$\mu_{H_r,\epsilon}$};
  \node at (2,-0.2) {$\rho$};
  \node at (-0.2,1.5) {$t$};
\end{tikzpicture}
\end{center}
\caption{Behavior of $\mu_{H_r,\epsilon}$ for $r > 1$.}
\label{fig:9}
\end{figure}

\begin{proof}Putting together all constraints gives
$$0 \leq nH_rJ_{n,r,\epsilon}(\rho) < \cosh^{n-r}(\rho).$$
Notice that $\rho_+^{\epsilon}$ is finite if and only if 
$$f_{\epsilon}(\rho) \coloneqq \cosh^{n-r}(\rho)-nH_rJ_{n,r,\epsilon}(\rho)$$
admits a zero. One has $f_{\epsilon}(\epsilon) > 0$, whereas the derivative of $f_{\epsilon}$ is
$$f_{\epsilon}'(\rho) = \frac{\cosh^{n-1}(\rho)}{\sinh^{r-1}(\rho)}((n-r)\tanh^r(\rho)-nH_r),$$
and is negative since $H_r > (n-r)/n$.
Moreover, $f_{\epsilon}'$ tends to $-\infty$, hence a zero $\rho_+^{\epsilon}$ exists and is unique.
For $n=r$, $f_{\epsilon}$ reduces to $1-nH_nJ_{n,\epsilon}$, which has a zero $\rho_+^{\epsilon} > \epsilon$ regardless of the value of $H_n > 0$.
The remaining details on the behavior of $\mu_{H_n,\epsilon}$ follow as in previous section.
\end{proof}

\begin{remark}
The technique used for Proposition \ref{prop:c2-reg} can be combined with \eqref{second-derivative-mu} and yields $C^2$-regularity of the translation $H_r$-hypersurface at points corresponding to $\rho = \rho_+^{\epsilon}$.
At points corresponding to $\rho = \epsilon$ when $r > 1$ we only have regularity $C^1$.
\end{remark}

By using the translation defined at the beginning of this section on the curves defined in Proposition \ref{curve-translation-prop}, one gets translation $H_r$-hypersurfaces in $\mathbb H^n \times \mathbb R$, which we describe in the following theorem.
Recall that $\pi$ is the totally geodesic hyperplane in $\mathbb H^n$ orthogonal to the plane containing the support of the curve given by the function $\mu_{H_r,\epsilon}$ at the origin.

\begin{theorem}
\label{str-thm-cyl2}
Let $r > 1$, $H_r > (n-r)/n$, and $\epsilon > 0$. 
Reflect the graph of $\mu_{H_r,\epsilon}$ on $[\epsilon,\rho_+^{\epsilon}]$ with respect to the horizontal slice $\mathbb H^n \times \{\mu_{H_r,\epsilon}(\rho_+^{\epsilon})\}$.
Translating the arc obtained along geodesics through the origin in $\pi$ gives an $H_r$-hypersurface with the topology of $\mathbb R^{n-1} \times [0,1]$ and of class $C^2$ away from the boundary.
The boundary components are planar equidistant hypersurfaces with distance $\epsilon$ from $\pi$, they lie in two different slices, and can be obtained from one another by a vertical translation.
\end{theorem}

\section{Estimates}
\label{estimates}

In this section we collect estimates that will be needed in the proof of Theorem \ref{main-thm}.
We define radii and heights related to pieces of the hypersurfaces classified in the previous sections, and study the interplay between them.
First we need to compare \emph{spheres} and \emph{horizontal cylinders}.

Fix $n \geq r > 1$ and $H_r > (n-r)/n.$ 
Denote by  $\mathcal S_r$ the sphere in $\mathbb H^n \times \mathbb R$ with $r$-th mean curvature $H_r$, namely the compact rotation hypersurface generated by the graph of $\lambda_{H_r,0}$ in Theorems \ref{str-thm2}, \ref{str-thm4}, \ref{str-thm6}, \ref{str-thm8}.  Let $R_{\mathcal S_r} \coloneqq \rho_+$, where $\rho_+$ was defined as the length of the domain of $\lambda_{H_r,0}$.

For any $\epsilon>0$, let us denote by $\mathcal C_{r,\epsilon}$ the $H_r$-hypersurface described in Theorem \ref{str-thm-cyl2}, which is a portion of a horizontal cylinder. 
Set $R_{\mathcal C_{r,\epsilon}} \coloneqq \rho_+^{\epsilon}-\epsilon$ where $\rho_+^{\epsilon}$ is the 
unique value such that 
$$f_{\epsilon}(\rho_+) = \cosh^{n-r}(\rho_+)-nH_rJ_{n,r,\epsilon}(\rho_+) = 0.$$ 
Note that $\mathcal C_{r,\epsilon}$ has a horizontal hyperplane of symmetry $P$ and  $R_{\mathcal C_{r,\epsilon}}$ is the distance between the projection of the boundary of $\mathcal C_{r,\epsilon}$ on $P$ and $\mathcal C_{r,\epsilon}\cap P.$ 

The next estimate will be used in Claim \ref{claim2} for the proof of Theorem \ref{main-thm}.
\begin{lemma} 
\label{radii-comparison}
For all $n$, $r$, $H_r$ with $n \geq r > 1$ and $H_r > (n-r)/n$, there exists a positive $\epsilon=\epsilon(n,r,H_r)$ such that $R_{\mathcal C_{r,\epsilon}} <R_{\mathcal S_r}$. 
\end{lemma}
\begin{remark}
A version of this statement for $r=1$ is given in Nelli--Pipoli \cite[Lemma~3.3]{nelli-pipoli}. Lemma \ref{radii-comparison} may be viewed as an extension of the latter to $r > 1$. 
\end{remark}
\begin{proof}
We have already shown that for $H_r > (n-r)/n$ (or $H_n > 0$) the function $f_{\epsilon}$ is decreasing.
Since $\rho_+^{\epsilon}>\epsilon$, we have $\lim_{\epsilon\rightarrow\infty}\rho_+^{\epsilon}=\infty$. 
 
Note that the function $\epsilon \mapsto \rho_+^{\epsilon}$ is continuous and increasing. 
To see this, let $0<a<b$ and $\rho > 0$, so that
$$J_{n,r,a}(\rho)=\int_a^b\frac{\cosh^{n-1}(x)}{\sinh^{r-1}(x)}\,dx+\int_{b}^{\rho}\frac{\cosh^{n-1}(x)}{\sinh^{r-1}(x)}\,dx>\int_{b}^{\rho}\frac{\cosh^{n-1}(x)}{\sinh^{r-1}(x)}\,dx=J_{n,r,b}(\rho).$$ 
It follows that
$$f_a(\rho_+^b)<\cosh^{n-r}(\rho_+^b)-nH_rJ_{n,r,b}(\rho_+^b)=f_b(\rho_+^b)=0=f_a(\rho_+^a).$$
Since $f_a$ is decreasing, $\rho_+^b>\rho_+^a$ holds.
 
We claim that $\rho_+^{\epsilon}<\sqrt\epsilon$ if $\epsilon$ is small enough. 
By definition of $\rho_+^{\epsilon}$ and the fact that $f_{\epsilon}$ is decreasing, it is enough to prove that $f_{\epsilon}(\sqrt\epsilon)<0$ for $\epsilon$ small enough.
Since the function $x\cosh(x)-\sinh(x)$ is positive for $x > 0$, we deduce that 
$$\frac{\cosh^{n-1}(x)}{\sinh^{r-1}(x)} = \cosh^{n-r}(x)\frac{\cosh^{r-1}(x)}{\sinh^{r-1}(x)} > \frac{1}{x^{r-1}},$$
whence
$$
J_{n,r,\epsilon}(\sqrt\epsilon)\geq\left\{\begin{array}{rcl}
-\frac12\log(\epsilon)&\text{ if }&r=2,\\
\frac{1}{r-2}(\epsilon^{{2-r}}-\epsilon^{\frac{2-r}{2}})&\text{ if }&r>2.
\end{array}\right.
$$ 
In both cases $\lim_{\epsilon\rightarrow 0}J_{n,r,\epsilon}(\sqrt\epsilon)=+\infty$. It follows that $f_{\epsilon}(\sqrt\epsilon)<0$ for $\epsilon$ sufficiently small, hence the claim is proved.
We deduce that $\epsilon<\rho_+^{\epsilon}<\sqrt\epsilon$ for $\epsilon$ small, so $\lim_{\epsilon\rightarrow 0}\rho_+^{\epsilon}=0$. 

Given $n \geq r > 1$ and $H_r > (n-r)/n$, the value of $R_{\mathcal S_r}$ is fixed. By the above statements, there is a value of $\epsilon > 0$ such that $R_{\mathcal C_{r,\epsilon}} < \rho_+^{\epsilon} < R_{\mathcal S_r}$.
\end{proof}

The next type of hypersurfaces we consider are \emph{annuli}.
Let  $n>r$, $H_r = (n-r)/n$, and choose $d_r > 0$. 
For these values of the parameters, the functions $\lambda_{H_r,d_r}$ for $r$ even and odd share the same behavior. 
Specifically, they have a zero $\rho_-$, which is the only solution of $\sinh^{n-r}(\rho) -(nH_rI_{n,r}(\rho)+d_r)=0$, and start with vertical tangent. 
After a vertical reflection across the slice $\mathbb H^n \times \{0\}$ and rotation about a vertical axis, each curve produces an unbounded annulus (see Theorem \ref{str-thm3}, \ref{str-thm7}).

Let us highlight a property of $d_r$ that will simplify our calculations.  
Since $nI_{n,r}(\rho) \approx \rho^n$ for $\rho$ close to $0$, for $\rho_-$ small we estimate
$$d_r = \sinh^{n-r}(\rho_-)-nH_rI_{n,r}(\rho_-) \approx \rho_-^{n-r}-H_r\rho_-^{n} = \rho_-^{n-r}(1-H_r\rho_-^r).$$
This implies 
\begin{equation}
\label{dr}
\lim_{d_r \to 0} \frac{d_r}{\rho_-^{n-r}} = 1.
\end{equation}
We need to consider the portion of the previous annulus between the slices $\mathbb H^n \times \{0\}$ and $\mathbb H^n \times \{h^*\},$ where $h^*$ is defined as 
\begin{equation}
\label{h-star}
h^* \coloneqq\int_{\rho_-}^{2\rho_-} \frac{((n-r)I_{n,r}(\xi)+d_r)^{\frac1r}}{\sinh^{\frac{n-r}{r}}(\xi)}\,d\xi.
\end{equation}
Observe that by \eqref{dr} we can interpret $h^*$ as an approximation of the value $\lambda_{{(n-r)}/{r},d_r}(2\rho_-)$ for $\rho_-$ small. Moreover $h^*< \lambda_{{(n-r)}/{r},d_r}(2\rho_-).$

Let now  $n=r.$  For $d_n > 0$ small enough, we consider portions of the peaked spheres found in Theorems \ref{str-thm4} and \ref{str-thm8}, so that the cases $n$ even and odd can be treated together. 
Here $\rho_-$ is not defined a priori, so we choose $\rho_- = d_n^{2/n}$ and define $h^*$ as follows
(by abuse of notation, we use the notation  $h^*$ as above)
\begin{equation}
\label{new-h}
h^* \coloneqq \int_{\rho_-}^{2\rho_-} d_n^{\frac1n} = d_n^{\frac3n}.
\end{equation}
Note that when $\rho_-$ is small,  then  $2\rho_-<\rho_+$  and $h^*$ is an approximation of 
$\lambda_{H_n,d_n}(2\rho_-)-\lambda_{H_n,d_n}(\rho_-)$, which is the height of the portion of the peaked sphere between two slices 
intersecting it in codimension one spheres of radii ${\rho_-}$ and ${2\rho_-}$. Moreover $h^*<\lambda_{H_n,d_n}(2\rho_-)-\lambda_{H_n,d_n}(\rho_-).$
 
For any $n \geq r$ we define $\rho_{H_r}^*$ implicitly as
\begin{equation}
\label{rho*}
h^* \eqqcolon \int_0^{\rho_{H_r}^*} \frac{(nH_rI_{n,r}(\xi))^{\frac1r}}{\sqrt{\sinh^{\frac{2(n-r)}{r}}(\xi)-(nH_rI_{n,r}(\xi))^{\frac2r}}}\,d\xi,
\end{equation}

Notice that $\rho_{H_r}^*$ is the radius of the intersection of the  sphere $\mathcal S_r$ of constant curvature $H_r$  with a slice at vertical distance $h^*$ from the South Pole. 

As above, we assume $\rho_-$ is small, which is equivalent to requiring $d_r$ small (recall that $d_r \to 0$ if and only if $\rho_- \to 0$).

\begin{lemma}
\label{various-estimates}
Let $n > r$, $d_r > 0$, $H_r = (n-r)/n$, take $h^*$ as in \eqref{h-star}, and $\rho_{H_r}^*$ as in \eqref{rho*}. Then
$$\lim_{d_r \to 0} \rho_{H_r}^* = 0,  \qquad \lim_{d_r \to 0} \frac{\rho_-}{\rho_{H_r}^*} = 0.$$
For $n=r$, $d_n > 0$, and $H_n > 0$, take $h^*$ as in \eqref{new-h} and $\rho_{H_n}^*$ as in \eqref{rho*}. Then 
$$\lim_{d_n \to 0} \rho_{H_n}^* = 0, \qquad \lim_{d_n \to 0} \frac{\rho_-}{\rho_{H_n}^*} = 0.$$
\end{lemma}

\begin{proof}

First assume $r<n.$
For $d_r$ small the right-hand side of \eqref{rho*} is approximated as
$$\int_0^{\rho_{H_r}^*} H_r^{\frac1r}\xi\, d\xi = \frac{H_r^{\frac1r}}{2}(\rho_{H_r}^*)^2.$$
We then approximate $h^*$  in \eqref{h-star} as
\begin{equation}
\label{approx-h}
h^* \approx \int_{\rho_-}^{2\rho_-} \left(\frac{n-r}{n}\xi^r+\frac{d_r}{\xi^{n-r}}\right)^{\frac{1}{r}}\,d\xi \approx \int_{\rho_-}^{2\rho_-} d_r^{\frac{1}{r}}\xi^{\frac{r-n}{r}}\,d\xi.
\end{equation}
Assume now $n \neq 2r$. We integrate \eqref{approx-h} to find
\begin{equation*}
h^* \approx \frac{r}{2r-n}(2^{\frac{2r-n}{r}}-1)\rho_-
\end{equation*}
On the other hand, $h^* \approx \frac{H_r^{\frac1r}}{2}(\rho_{H_r}^*)^2$. Since $d_r \to 0$ is equivalent to $\rho_- \to 0$, it is clear that $\lim_{d_r \to 0} \rho_{H_r}^* = 0$ and
$$\lim_{d_r \to 0} \frac{\rho_-}{\rho_{H_r}^*} = \lim_{d_r \to 0} \sqrt{\rho_-} = 0.$$
If $n = 2r$ we need to integrate \eqref{approx-h} in a different manner, namely
\begin{align*}
\int_{\rho_-}^{2\rho_-} d_r^{\frac1r}\xi^{-1}\,d\xi & = d_r^{\frac1r}\ln 2 \approx \rho_-\ln 2.
\end{align*}
Then again, $\lim_{d_r \to 0} \rho_-/\rho_{H_r}^* = 0$. 

When $n=r$ the proof is analogous provided that $h^* \coloneqq d_n^{\frac3n}$, as in \eqref{new-h}. 
\end{proof}

\section{Hyperbolic lima\c{c}on}
\label{limacon}

The goal of this section is to improve the estimates on the size of the hyperbolic lima\c con introduced in \cite{nelli-pipoli}. This hypersurface of $\mathbb H^n$ generalizes the well-known \emph{lima\c{c}on de Pascal} in the Euclidean plane, and it will play an important role in the proof of Theorem~\ref{main-thm}. We start by recalling its definition.

\begin{definition}
\label{limacon-def}
Let $A$ and $C$ be two distinct points in $\mathbb H^n$, and $c > 0$ be a constant. 
Let $\mathcal C$ be the geodesic sphere with radius $c$ centered at $C$.
For any $P \in \mathcal C$ define $A_P$ to be the reflection of $A$ across the totally geodesic hyperplane in $\mathbb H^n$ tangent to $\mathcal C$ at $P$.
The set 
$$\mathcal L \coloneqq \{A_P \in \mathbb H^n: P \in \mathcal C\}$$
is called \emph{hyperbolic lima\c{c}on}, and $A$ is called  \emph{base point} of $\mathcal L.$
\end{definition}
Since the hyperbolic space is two-points homogeneous, up to isometries of the ambient space $\mathcal L$ depends only on two parameters: $a \coloneqq d(A,C)$, where $d$ is the hyperbolic distance, and $c > 0$ as in Definition~\ref{limacon-def}. The shape of $\mathcal L$ changes depending on whether $a=c$, $a<c$, or $a>c$. Here we are only interested in the latter case. We refer to Nelli--Pipoli \cite[Section 2]{nelli-pipoli} for general properties of $\mathcal L$.

The following result improves \cite[Lemma 2.5]{nelli-pipoli} and will allow to remove the pinching assumption in 
\cite[Theorem 4.1]{nelli-pipoli}.

\begin{lemma}
\label{limacon-lemma}
Take $\mathcal{L}$ to be the hyperbolic lima\c{c}on with $a > c$ and base point $A$. Let $\mathcal{C}$ be the geodesic sphere defining $\mathcal L$, $C$ be its center, and $X$ be the point of $\mathcal C$ with minimal distance from~$A$.
Then $\mathcal L$ has two loops, one inside the other, and it has a self-intersection only at~$A$. Moreover the following statements hold. 
\begin{enumerate}
\item The smaller (resp.\ larger) loop of $\mathcal L$ is contained in (resp.\ contains) the disk centered at $X$ and radius $a-c$.
\item The smaller loop of $\mathcal L$ bounds the disk centered at $X$ and radius 
\begin{equation}\label{improvement-distance}
\ell(a,c) \coloneqq \cosh^{-1}\left(\cosh(a-c)-\frac{\sinh c}{2\sinh a} \sinh^2(a-c) \right),
\end{equation}
\item All of $\mathcal L$ sits inside the disk centered at $C$ and radius $a+2c$.
\end{enumerate} 
\end{lemma}

\begin{proof}
Since $a>c$, $\mathcal L$ has two loops, one inside the other, and has a self-intersection only at $A$, cf.\ \cite[Lemma 2.4]{nelli-pipoli}. The estimates (1) and (3) have been proved in \cite[Lemma 2.5]{nelli-pipoli}. It remains to prove (2).

Since $\mathcal L$ is invariant with respect to rotations about the geodesic passing through $A$ and $C$, we can assume $n=2$. We start by giving an explicit parametrization of $\mathcal L$ in the hyperboloid model for the hyperbolic space canonically embedded in the Minkowski space $\mathbb R^{2,1} = (\mathbb R^3,q)$, where $q$ is the standard scalar product of signature $(2,1)$.
Without loss of generality, we can assume that $A = (\sinh a, 0, \cosh a)$, and the center of $\mathcal C$ to be $(0,0,1)$.
Then we parametrize $\mathcal C$ by
$$\alpha(\theta) = (\sinh c \cos \theta, \sinh c\sin \theta, \cosh c).$$
Let $P = \alpha(\theta)$ for some $\theta$. We want to find the unique geodesic $\gamma_P$ through $P$ tangent to $\mathcal{C}$ explicitly: $\gamma_P$ is the geodesic passing through $P$ and generated by the unit tangent vector to $\mathcal{C}$ at $P$, which is
$$T(\theta) = \frac{\alpha'(\theta)}{\sqrt{q(\alpha'(\theta),\alpha'(\theta))}} = (-\sin \theta,\cos \theta,0).$$
Therefore $\gamma_P = \mathbb H^2 \cap \Pi_P$, where $\Pi_P$ is the plane in $\mathbb R^{2,1}$ passing through $O$, $P$, and parallel to $T$.
A unit normal to $\Pi_P$ with respect to $q$ is the vector 
$$\nu(\theta) = (\cosh c \cos \theta, \cosh c \sin \theta, \sinh c).$$
Following Definition \ref{limacon-def}, we need to reflect $A$ across $\gamma_P$. Since the reflection in $\mathbb H^2$ across $\gamma_P$ is the restriction to $\mathbb H^2$ of the reflection in $\mathbb R^{2,1}$ across $\Pi_P$, it follows that $\mathcal L$ can be parametrized as
$$
L(\theta) = A-2q(A,\nu(\theta))\nu(\theta).
$$
The point of $\mathcal C$ at minimal distance from $A$ is $X = (\sinh c, 0, \cosh c)$. Since $a > c$, then $X$ is in the compact region bounded by the smaller loop of the hyperbolic lima\c{c}on.
The strategy now is to compute the distance between $X$ and $\mathcal L$, then the smaller loop of $\mathcal{L}$ will bound a disk centered at $X$ and radius the above distance.
It is well known that the hyperbolic distance in the upper hyperboloid is
$$
d(A,B) = \cosh^{-1}(-q(A,B)), \qquad A,B \in \mathbb H^2.
$$
In order  to find the critical points of the function $\theta\mapsto d(X, L(\theta))$,
it is enough to find the critical points of the function $\theta\mapsto q(X,L(\theta))$.
We have
$$q(X,{L}(\theta)) = -\cosh(a-c) + q(\theta) \sinh 2c(1-\cos \theta),$$
where $q(\theta) \coloneqq q(A,\nu(\theta)) = \cosh c \sinh a \cos \theta - \sinh c \cosh a$.
Explicit computations give 
$$\frac{d}{d\theta}\left(q(X,{L}(\theta))\right) = \sinh 2c \sin \theta(2\sinh a \cosh c \cos \theta - \sinh(a+c)).$$
Hence critical points are given by 
\begin{equation}
\label{hyper-rels}
\sin \theta = 0, \qquad \text{ and } \qquad \cos\theta = \frac{\sinh (a+c)}{2\sinh a \cosh c}.
\end{equation}
The case $\theta = 0$ yields a new proof of \cite[Lemma 2.5, part 1]{nelli-pipoli}.
The case $\theta = \pi$ produces a disk centered at $X$ and radius $a+3c$, which is worse than the disk in \cite[Lemma 2.5, part 3]{nelli-pipoli}.
The case of interest is now the last one.
Let $\theta_0$ be such that $\cos \theta_0$ satisfies the second identity in \eqref{hyper-rels}. Then
$$q(X,{L}(\theta_0)) = -\cosh(a-c)+\frac{\sinh c}{2\sinh a}\sinh^2(a-c).$$
We then have
\begin{equation*}
d(X,\mathcal{L}) = \ell(a,c) = \cosh^{-1}\left(\cosh(a-c)-\frac{\sinh c}{2\sinh a} \sinh^2(a-c) \right),
\end{equation*}
hence the smaller loop of $\mathcal{L}$ bounds a disk of center $X$ and radius $\ell(a,c)$.
\end{proof}

We conclude this section with a list of properties of $\ell$ which will be useful for the estimates in the proof of Theorem \ref{main-thm}. 

\begin{lemma}
\label{limacon-rmk}
The following properties hold.
\begin{enumerate}
\item \label{item:lim1} For any $a>c \geq 0$, $\ell(a,0) = a$, $\ell(a,a) = 0$, and $\ell(a,c) > 0$. Moreover $\ell(a,c) < a-c$.
\item \label{item:lim2} The function $(a,c)\mapsto\ell(a,c)$ with domain $\{(a,c)\in\mathbb R^2 : a>c>0\}$ is increasing in the first variable and decreasing in the second one.
\item \label{item:lim3} For any $x>0$, then $\ell(4x,2x)>x$.
\end{enumerate}
\end{lemma}
\begin{proof}
The properties in (\ref{item:lim1}) follow directly by the definition of $\ell$, cf.\ \eqref{improvement-distance}.

As for (\ref{item:lim2}), we have
\begin{align*}
\frac{\partial}{\partial a}\cosh(\ell(a,c))&=\frac{\sinh(a-c)}{2\sinh^2a}\left(2\sinh^2a-\sinh^2c-\cosh(a-c)\sinh a\sinh c\right)\\
&>\frac{\sinh(a-c)}{2\sinh a}\left(\sinh a-\cosh(a-c)\sinh c\right)\\
&=\frac{\sinh^2(a-c)\cosh c}{2\sinh a}>0,
\end{align*}
where we have used the fact that $a>c>0$. Likewise
\begin{align*}
\frac{\partial}{\partial c}\cosh(\ell(a,c))&=-\frac{\sinh(a-c)}{2\sinh a}(2\sinh a + \cosh c \sinh(a-c)-2\sinh c \cosh(a-c)) \\
& = -\frac{3\sinh^2(a-c)\cosh c}{2\sinh a} < 0.
\end{align*}
Since the functions $\sinh$ and $\cosh$ are increasing in $[0,+\infty)$, the claim follows.

Let us now prove (\ref{item:lim3}). By \eqref{improvement-distance} we have 
\begin{align*}
\ell(4x,2x) & =\cosh^{-1}\left(\cosh(2x)-\frac{\sinh^2(2x)}{4\cosh(2x)}\right)\\
& =\cosh^{-1}\left(\frac{3(2\cosh^2 x-1)^2+1}{4(2\cosh^2x-1)}\right).
\end{align*}
It follows that $\ell(4x,2x)>x$ if and only if $$\frac{3(2\cosh^2 x-1)^2+1}{4(2\cosh^2x-1)}>\cosh x,$$ namely $(\cosh(x)-1)(\cosh^2 x + (\cosh x-1)(3\cosh^2 x + 3\cosh x + 1))>0$.
The latter holds true for all $x>0$, and we are done.
\end{proof}

\section{Ros--Rosenberg type theorem}
\label{main-result}

The second goal of the present paper is to prove a topological result about compact connected $H_r$-hypersurfaces embedded in $\mathbb H^n \times\mathbb R$ with planar boundary. This is a generalization of the classical result of Ros and Rosenberg \cite{ros-rosenberg} about the topology of constant mean curvature surfaces in the Euclidean three-dimensional space. 

\begin{theorem}
\label{main-thm}
Let $M$ be a compact connected hypersurface embedded in $\mathbb H^n \times [0,\infty) \subset \mathbb H^n \times \mathbb R$ with boundary a closed horoconvex  $(n-1)$-dimensional hypersurface $\Gamma$ embedded in the horizontal slice $\mathbb H^n \times \{0\}$. Assume $M$ has constant $r$-th mean curvature $H_r > (n-r)/n$ for some $r = 1,\dots, n$.
Then there is a constant $\delta = \delta(n,r,H_r) > 0$ small enough such that, if $\Gamma$ is contained in a disk of radius $\delta$, then $M$ is topologically a disk. 
\end{theorem}

We recall that a hypersurface $\Gamma$ of the hyperbolic space is called \emph{horoconvex} if all its principal curvatures are larger than one.

\begin{remark}
Let us make a few observations.
\begin{enumerate}
\item When $r=1$, Theorem \ref{main-thm} improves \cite[Theorem 4.1]{nelli-pipoli}. In fact, thanks to the new estimates given in Lemma \ref{limacon-lemma}(2), we do not need to assume any pinching on $\Gamma$. 
\item Elbert--Sa Earp \cite[Theorem 7.7]{elb-earp} proved that when $n>r$ and $0<H_r \leq (n-r)/n$, 
then a compact connected $H_r$-hypersurface $M$ embedded in $\mathbb H^n \times [0,\infty)$ with horoconvex boundary $\Gamma$ in the slice $\mathbb H^n \times \{0\}$ is necessarily a graph over the compact planar domain bounded by $\Gamma$.
In particular $M$ is a disk. 
Therefore, we focus on the cases $n>r$, with $H_r > (n-r)/n$, and $n=r$, with $H_n > 0$.
\item By using Alexandrov reflections with respect to vertical hyperplanes, we can show that $M$ shares the same symmetries of its boundary. In particular, when $\Gamma$ is a geodesic sphere, $M$ is rotationally symmetric.
It follows that $M$ is a portion of one of the compact hypersurfaces classified in Section \ref{rotational-surfaces}, and Theorem \ref{main-thm} is proved in this special case. 
\end{enumerate}
\end{remark}

In view of the previous remark, we will assume throughout that $\Gamma$ is not a geodesic sphere.

\begin{remark}
\label{convex-point}
In the following we will do extensive use of the tangency principle for $H_r$-hypersurfaces as it is stated in 
\cite[Theorem 1.1]{fontenele-silva}. In order to satisfy the assumptions there,  it is enough that the  hypersurface $M$ 
in Theorem \ref{main-thm} has a strictly convex point. This is guaranteed by \cite[Lemma 7.5]{elb-earp}.
\end{remark}

\begin{notations*}
Let us introduce some notations that will be useful in the proof of Theorem~\ref{main-thm}. 
For the  reader's convenience, there is a list of notations at the end of the article.
We denote by $\Omega$ the compact domain of $\mathbb H^n\times\{0\}$ bounded by $\Gamma$ and by $W$ the compact domain in $\mathbb H^n \times \mathbb R$ with boundary $M\cup\Omega$.
Given $n \geq r \geq 2$, and $H_r > (n-r)/n$, we fix an $\epsilon > 0$ such that Lemma \ref{radii-comparison} is satisfied. Denote by $\mathcal C_r \coloneqq \mathcal C_{r,\epsilon}$ the corresponding translation $H_r$-hypersurface of Theorem \ref{str-thm-cyl2}.
When $r=1$ we use the same notation, however recall that no choice of $\epsilon$ is involved.
Let $h_{\mathcal C_r}$ denote the height of $\mathcal C_r$ (namely $2\mu_{H}(\rho_+)$ for $r=1$ and $2\mu_{H_r,\epsilon}(\rho_+^{\epsilon})$ for $r > 1$).
Analogously let $h_{\mathcal S_r}$ be the height of $\mathcal S_r$ (i.e.\ $2\lambda_{H_r,0}(\rho_+)$, cf.\ Theorems \ref{str-thm2}, \ref{str-thm4}, \ref{str-thm6}, \ref{str-thm8}). 
We define $h_M$ to be the height of $M$ with respect to the slice $\mathbb H^n \times \{0\}$. 
The \emph{exterior} (resp.\ \emph{interior}) \emph{radius} of $\Gamma$ is the smaller (resp.\ larger) radius $\rho$ such that for any $p \in \Gamma$ there is a geodesic sphere $S$ with radius $\rho$ tangent to $\Gamma$ at $p$ and $\Gamma$ sits in (resp.\ encloses) the closed ball with boundary $S$. We write $r_{ext}$ for the exterior radius and $r_{int}$ for the interior one. Clearly $r_{ext} \geq r_{int}$, and equality occurs if and only if $\Gamma$ is a geodesic sphere. Moreover, since $\Gamma$ is horoconvex, $r_{int}$ and $r_{ext}$ are determined by the maximum and the minimum of the principal curvatures of $\Gamma$.
Finally we denote by $D(R)$ any disk of radius $R > 0$ in a horizontal slice of $\mathbb H^n \times \mathbb R$.
\end{notations*}

The strategy of the proof of Theorem \ref{main-thm} is similar to that of  \cite{nelli-pipoli}: if the height of $M$ is less than the height of $\mathcal C_r$, then $M$ is a graph over $\Omega$, otherwise it is a union of hypersurfaces, each one a graph over a suitable domain. As in \cite{ros-rosenberg}, at the end of the proof it will be clear that the union of such graphs has the topology of the disk. The hyperbolic lima\c{c}on described in Section \ref{limacon} will be used in various estimates.

\begin{lemma}
\label{lemma5}
Let $M$ and $\Gamma$ satisfy the assumptions of Theorem \ref{main-thm}. There is a disk $D(r_{min})$ in $\mathbb H^n \times \{0\}$
such that $M \cap (D(r_{min}) \times \mathbb R)$ is a graph, and $\ell(r_{ext},r_{ext}-r_{int}) \leq r_{min} < r_{int}$.
In particular, $r_{min}$ depends only on the principal curvatures of $\Gamma$.
\end{lemma}
\begin{proof}
In order to prove the statement, we apply Alexandrov's reflection technique with horizontal hyperplanes coming down from above.
Since $M$ is compact, the slice $\mathbb H^n \times \{t\}$, $t > h_M$ does not intersect $M$. 
Then we let $t$ decrease. When $t<h_M$, reflect the part above the slice and stop if there is a first contact point between $M$ and its reflection.
If we can get to $t=0$ without having contact points, then $M$ is a graph over $\Omega$ and we can choose $r_{min}< r_{int}.$
If this does not happen, there will be a $0 < t_0 < h_M/2$ such that the reflected hypersurface touches $M$ for the first time. 
If the intersection point lied in the interior of $M$ we would have a contradiction with the Maximum Principle, hence a first touching point belongs to $\Gamma$.
Let $q$ be one of such points.
Then the line $\{q\} \times (0,\infty)$ intersects $M$ exactly once, and $\{q\} \times (0,2t_0)$ is contained in the interior of $W$, as $t_0 < h_M/2$.
Note that the portion of $M$ above $\mathbb H^n \times \{t_0\}$ is a graph.

We now perform Alexandrov's reflections with respect to vertical hyperplanes, i.e.\ the product of a totally geodesic hypersurface of $\mathbb H^n$ and $\mathbb R$. Let $Q$ be one of such hyperplanes. Since $M$ is compact, we can assume that $Q\cap M=\varnothing$. Fix a point $x\in Q$ and let $\gamma$ be the geodesic passing through $x$ and orthogonal to $Q$. Move $Q$ along $\gamma$ towards $M$ such that $Q$ is always orthogonal to $\gamma$. By abuse of notation, we call again $Q$ any parallel translation of the initial hyperplane. When $Q$ touches $M$ for the first time, keep moving $Q$  and start reflecting through $Q$ the part of $M$ left behind $Q$. In order not to have a contradiction with the Maximum Principle, we can continue this procedure with no contact points between $M$ and its reflection until $Q$ enters $\Gamma$ at distance at least $r_{int}$ from it. 

We can avoid the dependence on the contact point $q$ by stopping reflecting when $Q$ is tangent to $\mathcal{C}$, where $\mathcal{C}$ is as follows.
Denote by $\mathcal{C}_{ext}$ the geodesic sphere in $\mathbb H^n \times \{0\}$ of radius $r_{ext}$, tangent to $\Gamma$ at $q$, and enclosing $\Gamma$.
Then $\mathcal{C}$ is the geodesic sphere with the same center as that of $\mathcal{C}_{ext}$ and radius equal to $r_{ext}-r_{int}$. 

Define $\mathcal{L}$ to be the set of the reflections of $q$ through any vertical hyperplane tangent to $\mathcal C$. It follows that $\mathcal{L}$ is a hyperbolic lima\c{c}on as in Definition~\ref{limacon-def} whose base point is $q$ and whose parameters are $a = r_{ext}$ and $c = r_{ext}-r_{int}$.
Since $a > c$, $\mathcal L$ has two loops. Moreover, since $\Gamma$ is horoconvex, the smaller loop of $\mathcal{L}$ sits in $\Omega$.

Furthermore, since $\{q\}\times\mathbb R$ intersects $M$ in exactly one point, then the same holds true for any $p$ in the compact planar domain bounded by the smaller loop of $\mathcal{L}$.
Define $r_{min}$ as the largest radius of a ball bounded by the smaller loop of $\mathcal L$. Then $M\cap(D(r_{min})\times\mathbb R)$ is a graph. 
Finally Lemma \ref{limacon-lemma} and Lemma \ref{limacon-rmk} imply $\ell(r_{ext},r_{ext}-r_{int}) \leq r_{min} < r_{int}$ at once.
We remark that $r_{min}$ depends only on $a$ and $c$, namely only on the curvature of $\Gamma$, but not on $q$.
\end{proof}

\begin{proof}[Proof of Theorem \ref{main-thm}]
We  first assume $h_M < h_{\mathcal C_r}$.

Recall that $R_{\mathcal C_r} = \rho_+^{\epsilon}-\epsilon$ for $r > 1$, and $R_{\mathcal C_1} = \rho_+$.
We can then adapt the proof as in Nelli--Pipoli \cite{nelli-pipoli} to our case.

\begin{customthm}{\RomanNumeralCaps{1}}
\label{claim1}
The hypersurface $M$ lies in $D(r_{ext}+R_{\mathcal C_r}) \times [0,h_{\mathcal C_r}).$
\end{customthm}
\begin{proof} Consider the $H_r$-hypersurface $\mathcal C_r$. Its lower boundary is in the slice $\mathbb H^n \times \{0\}$ and the upper boundary sits in the slice $\mathbb H^n \times \{h_{\mathcal C_r}\}$.
We call $\mathcal C_{r}$ any horizontal translation or rotation of $\mathcal C_{r}$.
Since $M$ is compact, we can translate $\mathcal C_{r}$ horizontally so that $M \cap \mathcal C_{r} = \varnothing$ and $M$ lies in the part of $\mathcal C_{r}$ containing the axis of $\mathcal C_r$. 
Then we move $\mathcal C_{r}$ isometrically towards $M$ until $\mathcal C_{r}$ touches $M$ for the first time (see Figure \ref{fig-cylinder}).
By the Maximum Principle, $\mathcal C_{r}$ and $M$ do not touch at any interior point.
Since $h_M < h_{\mathcal C_r}$, the first touching point belongs to $\Gamma$.
The same steps can be repeated for $\mathcal C_{r}$ with any horizontal axis. By definition of $r_{ext}$ we get that $M$ sits inside $D(r_{ext}+R_{\mathcal C_r}) \times [0,h_{\mathcal C_r})$.
\end{proof}

\begin{figure}[htbp]
\centerline{\includegraphics[scale=0.15]{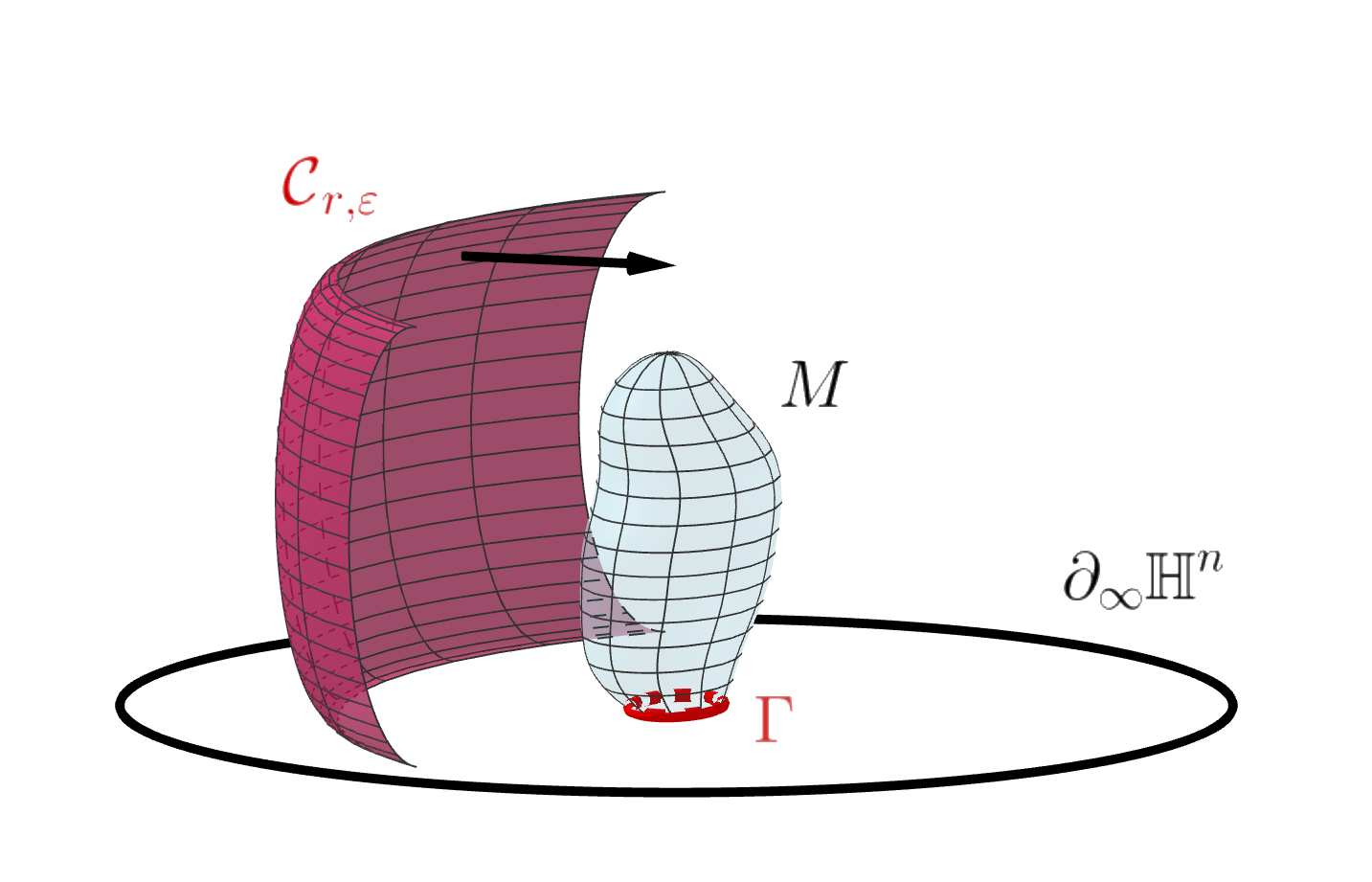}}
\caption{{Here is a representation for $n=2$. 
The black circle is the boundary at infinity of the slice $\mathbb H^2 \times \{0\}$,  $M$ is  the blue surface  whose  red boundary  $\Gamma$ lies in $\mathbb H^2 \times \{0\}.$ One moves
the purple half-cylinder  $\mathcal C_{r,\epsilon}$ isometrically towards $M$. Note that $M$ may have non-trivial topology.}}
\label{fig-cylinder}
\end{figure}

\begin{customthm}{\RomanNumeralCaps{2}}
\label{claim2}
If $\Gamma$ is sufficiently small, then $M$ is contained in the cylinder $\Omega \times \mathbb R$.\end{customthm}
\begin{proof}
By Lemma \ref{radii-comparison} and our choice of $\epsilon$ one has $R_{\mathcal C_r} < R_{\mathcal S_r}.$ 
Recall that $\mathcal S_r$ is the sphere with the same $r$-th mean curvature as that of $M$. Cut $\mathcal S_r$ with its horizontal hyperplane of symmetry and let $\mathcal S_r^+$ be the upper hemisphere. 
Now take $\Gamma$ small enough so that $R_{\mathcal C_r} + r_{ext} < R_{\mathcal S_r}$.
Translate $\mathcal S_r^+$ horizontally in such a way that the intersection of its axis of rotation with the slice $\mathbb H^n \times \{0\}$ coincides with the center of the disk found in Claim \ref{claim1}.
Translate upwards  $\mathcal S_r^+$ such that $\mathcal S_r^+\cap M=\varnothing$. 
By the Maximum Principle, Claim \ref{claim1}, and the hypothesis on $\Gamma$, we can translate $\mathcal S_r^+$ downwards without having a contact point between $\mathcal S_r^+$ and $M$ until the boundary of $\mathcal S_r^+$ is contained in the slice $\mathbb H^n \times \{0\}$, whence $M$ is below $\mathcal S_r^+$ (see Figure \ref{fig-half-sphere}).

\begin{figure}[htbp]
\centerline{\includegraphics[scale=0.15]{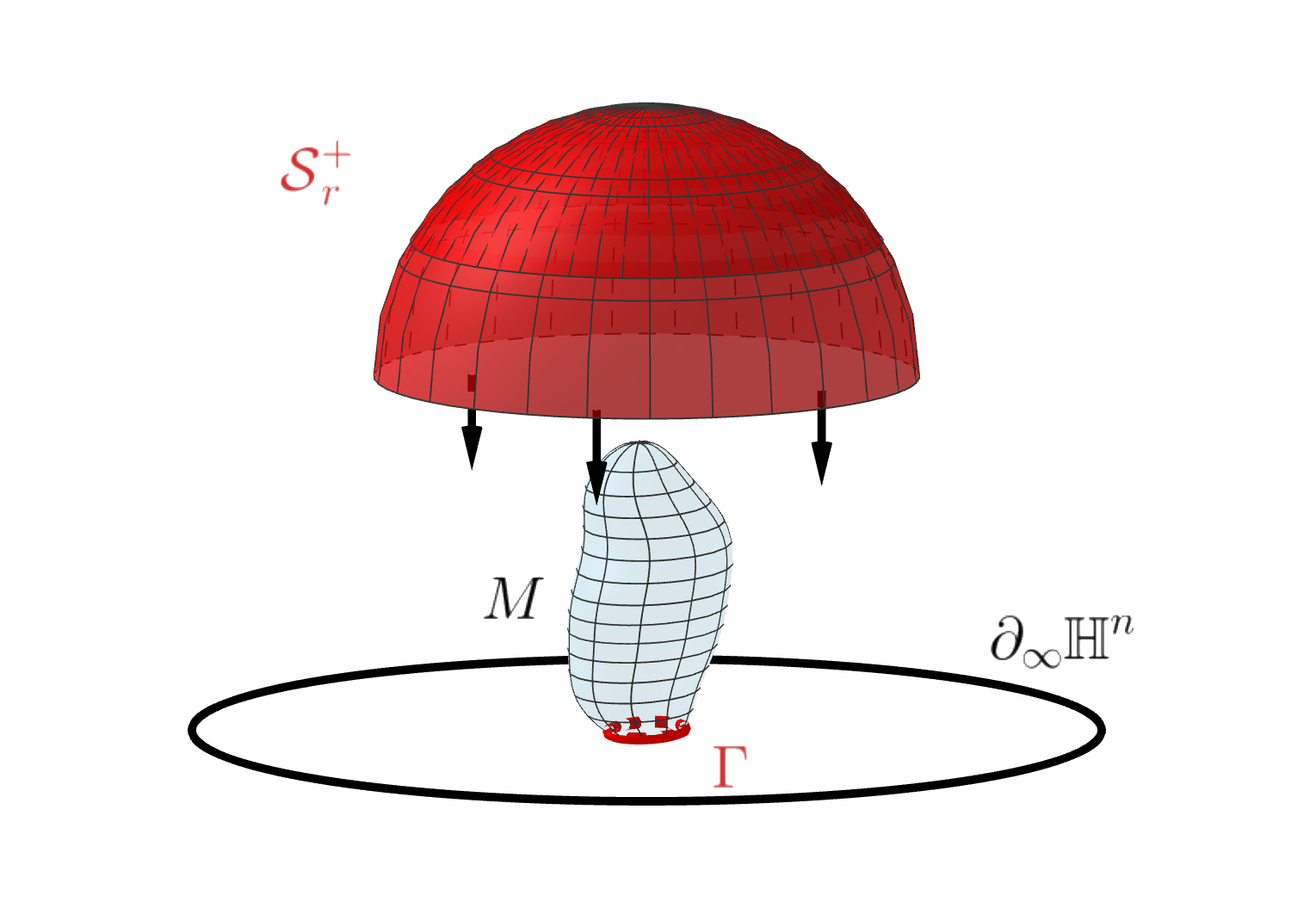}}
\caption{The red surface is a spherical cap $S_r^+$ moving downwards isometrically.}
\label{fig-half-sphere}
\end{figure}

By the Maximum Principle and the fact that $r_{ext} < R_{\mathcal S_r}$, one can translate horizontally $\mathcal S_r^+$ without having a contact point with $M$ until $\mathcal S_r^+$ becomes tangent to $\Gamma$ at any point of $\Gamma$, which gives the claim.
\end{proof}

\begin{customthm}{\RomanNumeralCaps{3}}
\label{claim3}
The hypersurface $M$ is a graph over $\Omega,$ hence it is a disk.
\end{customthm}
\begin{proof}
By Alexandrov's reflections technique with horizontal hyperplanes coming down from above, it follows that $M$ is a graph over $\Omega$, which proves Theorem \ref{main-thm} when $h_M < h_{\mathcal C_r}$.
Observe that $\delta$ can be taken as $R_{\mathcal S_r}-R_{\mathcal C_r}$, cf.\ Claim \ref{claim2}.
\end{proof}

We now assume that $h_M \geq h_{\mathcal C_r}$. 
Alexandrov's reflection technique with horizontal and vertical hyperplanes guarantees that the part of $M$ above the plane $t = h_M/2$ is a graph over a domain of $\mathbb H^n\times\{0\}$ and that the part of $M$ outside the cylinder $\Omega \times \mathbb R$ is a graph over a domain of $\Gamma\times\mathbb R$.
The goal is to prove that $M$ is the union of such graphs, i.e.\ $M \cap (\Omega \times (0,h_M/2])$ is empty.
In this way it will be clear that $M$ has the topology of a disk (see Figure \eqref{Figure-theorem-3}).

\begin{figure}[htbp]
\centerline{\includegraphics[scale=0.15]{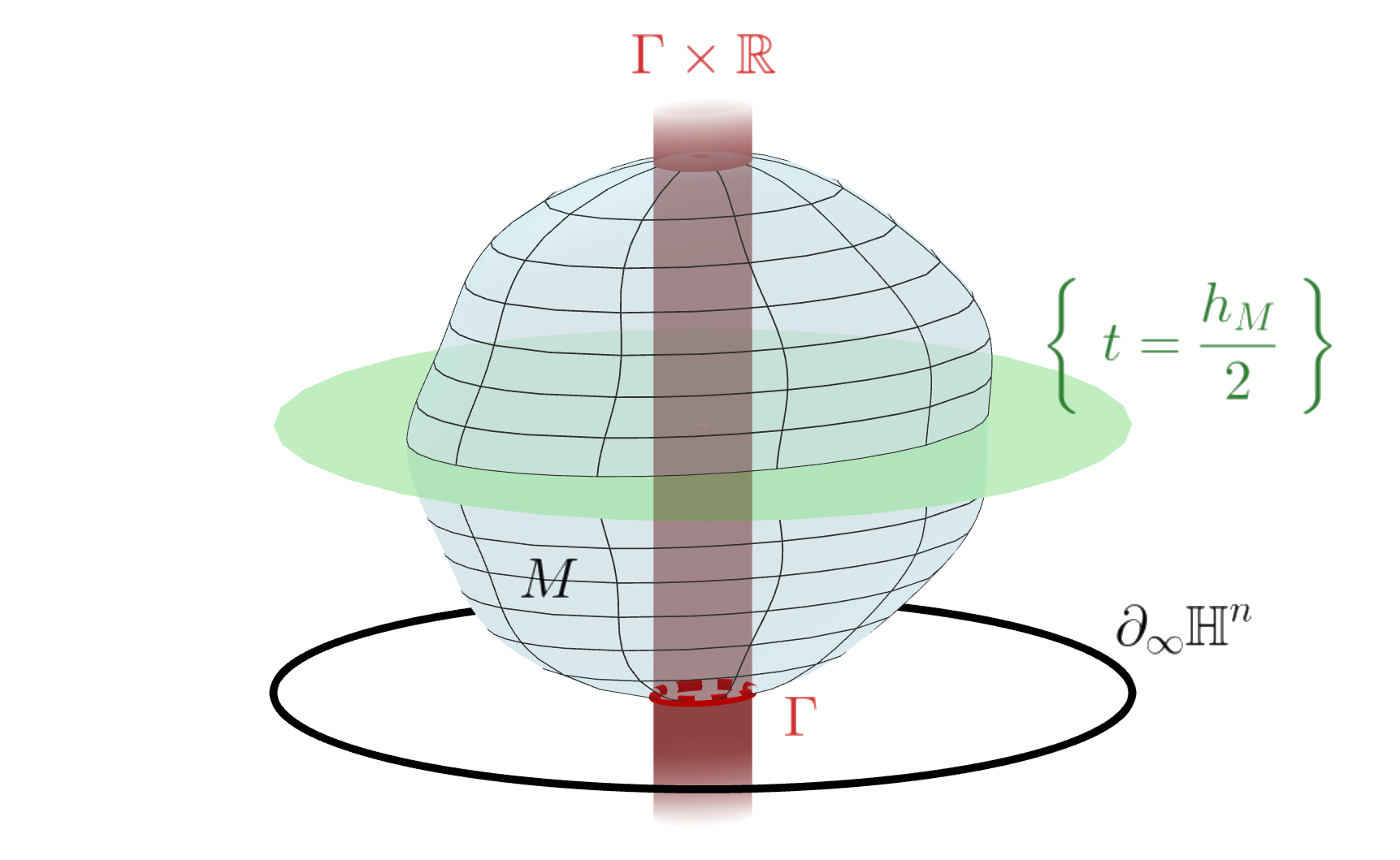}}
\caption{Decomposition of $M$: the part of $M$ above the green hyperplane and the part of $M$ outside the red cylinder are graphs.}
\label{Figure-theorem-3}
\end{figure}
 
Recall the definition of $h^*$ in \eqref{h-star} for $n>r$ and \eqref{new-h} for $n=r$.
Hereafter we show that $\Omega \times [h^*,h_M/2]$ contains no point of $M$ if $\Gamma$ is small enough, 
and lastly we prove that there is no interior point of $M$ in $\Omega \times [0,h^*]$ as well.

Before doing this we discuss how the various quantities we use are related to one another.
Let $d_r > 0$ be such that
$$\ell(r_{ext},r_{ext}-r_{int}) \leq \rho_- \leq r_{min} < r_{int} < r_{ext},$$
where $r_{min}$ is the radius defined in Lemma \ref{lemma5} and $\rho_-$ is the minimum of the interval where $\lambda_{(n-r)/n,d_r}$ is defined when $n>r$ (see Section~\ref{rotational-surfaces}), and for $n=r$ was chosen in Section \ref{estimates} to be $d_n{}^{2/n}$.
Note that if $r_{ext} \to 0$, i.e.\ $\Gamma$ shrinks to a point, then $d_r \to 0$,
and so $\rho_-$, $h^*$, and $\rho_{H_r}^*$ go to zero as well (cf.\ Lemma \ref{various-estimates}).
Hence if $r_{ext}$ is small enough, then $h^* \ll\frac{h_M}{2}$.
Further, since $\ell(r_{ext},r_{ext}-r_{int}) > 0$, we can find $\alpha > 0$ such that 
\begin{equation}\label{alpha}
\alpha r_{ext} < \ell(r_{ext},r_{ext}-r_{int}) \leq \rho_-,
\end{equation}
whence $\rho_{H_r}^*/r_{ext} > \alpha \rho_{H_r}^*/\rho_-$. Taking $\Gamma$ small enough, by Lemma \ref{various-estimates} we have
$$\frac{\rho_{H_r}^*}{r_{ext}} > \alpha \frac{\rho_{H_r}^*}{\rho_-} > 3,$$
therefore we can assume
\begin{equation}
\label{last-estimate}
\rho_{H_r}^* > 3r_{ext}.
\end{equation}

\begin{customthm}{\RomanNumeralCaps{4}}
\label{claim4}
The compact domain bounded by $M \cap (\mathbb H^n \times \{h_M - h^*\})$ contains a geodesic segment of length at least $\rho_{H_r}^*$.
\end{customthm}
\begin{proof}
Alexandrov's technique with respect to horizontal hyperplanes implies that the reflection of points in $M$ at height $h_M$ across the hyperplane $\mathbb H^n \times \{h_M/2\}$ sits in the closure of $\Omega$. 
We can assume that one of these points lies on the $t$-axis after applying a horizontal isometry.
Let $M'$ be the portion of $M$ above the hyperplane $\mathbb H^n \times \{h_M - h^*\}$. 
Then $M'$ is a graph with height $h^*$.
Suppose that for any $p \in \partial M'$ the distance between $p$ and the $t$-axis is smaller than $\rho_{H_r}^*.$
Cut $\mathcal S_r$ with a horizontal hyperplane so that the spherical cap  $S_r'$ above that hyperplane has height $h^*$.
Then translate $\mathcal S_r'$ up until it has empty intersection with $M$, then move it downwards. 
The Maximum Principle implies there is no contact point between $\mathcal S_r'$ and the interior of $M'$ at least until the boundary of $\mathcal S_r'$ reaches the level $t = h_M - h^*$.
Therefore the height of $M'$ is less than $h^*$, which is a contradiction.
\end{proof}

\begin{customthm}{\RomanNumeralCaps{5}}
\label{claim5}
The domain bounded by $M \cap (\mathbb H^n \times \{h_M - h^*\})$ contains a disk $D(R)$ with $R > \ell(\rho_{H_r}^*-r_{ext},r_{ext}).$
\end{customthm}
\begin{proof}
Up to horizontal translation, we can assume that one of the endpoints of the geodesic segment found in Claim \ref{claim4} is on the $t$-axis. Let $p$ be the other endpoint. Consider a geodesic sphere $\mathcal C_{ext}$ of $\mathbb H^n\times\{0\}$ tangent to $\Gamma$ and containing $\Gamma$.  Reflect the point $p$ across any vertical hyperplane tangent to $\mathcal C_{ext}$ in $\mathbb H^n \times \mathbb R$.

The set of such reflections is a hyperbolic lima\c{c}on $\mathcal L$ in $\mathbb H^n \times \{h_M-h^*\}$ with base point $p$ (see Figures 
\ref{fig:limacon} and \ref{fig:limacon-flat}).
By the choice of $p$, the parameters of $\mathcal L$ are $a > \rho_{H_r}^*-r_{ext}$ and $c = r_{ext}$.
By \eqref{last-estimate}, $a>c$, so $\mathcal L$ has two loops, and the smaller one is contained in $W$ -- argue as in Lemma \ref{lemma5}.
The claim now follows by Lemma \ref{limacon-lemma}.
\end{proof}

\begin{figure}
    \centering
    \begin{minipage}{0.50\textwidth}
        \centering
        \includegraphics[width=0.9\textwidth]{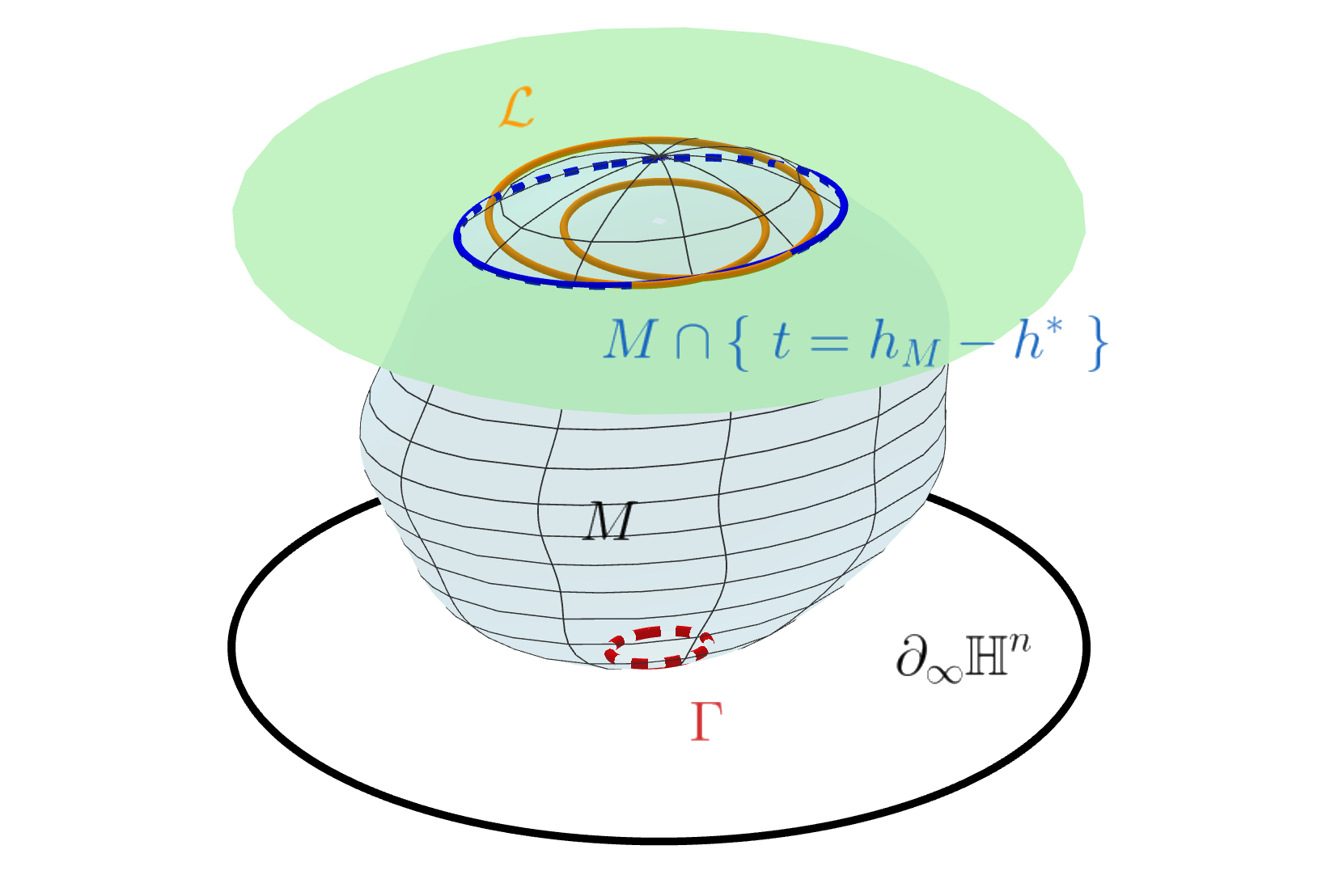} 
        \caption{ 
The orange curve is the hyperbolic lima\c{c}on.}\label{fig:limacon}
    \end{minipage}\hfill
    \begin{minipage}{0.50\textwidth}
        \centering
        \includegraphics[width=0.9\textwidth]{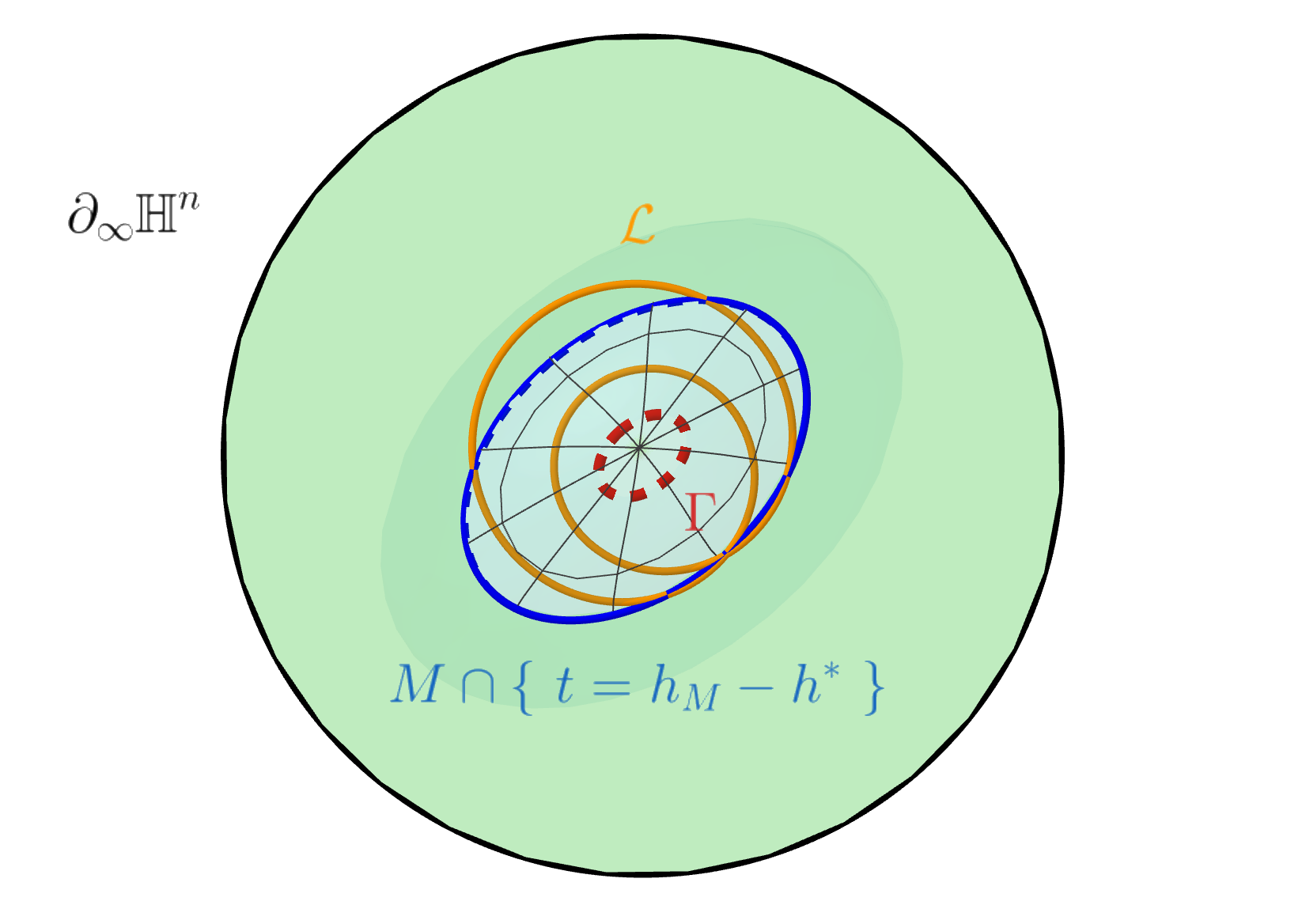} 
        \caption{The hyperbolic lima\c{c}on from above.}
        \label{fig:limacon-flat}
    \end{minipage}
    \end{figure}



\begin{customthm}{\RomanNumeralCaps{6}}
\label{claim6}
The intersection between $M$ and $D(R) \times [h^*,h_M-h^*]$ is empty.
\end{customthm}
\begin{proof}
Claim \ref{claim5} implies that $D(R)$ is contained in $W$, and since we have chosen $h^* \ll h_M$,
the hyperplane $\mathbb H^n \times \{h_M-h^*\}$ is above the hyperplane $\mathbb H^n \times \{h_M/2\}$. 
By applying the Alexandrov's reflection technique with horizontal hyperplanes, the reflection of $D(R)$ across $\mathbb H^n \times \{\tau\}$ is contained in $W$ for all $\tau \in [h_M/2,h_M-h^*]$. 
The claim then follows.
\end{proof}

\begin{customthm}{\RomanNumeralCaps{7}}
\label{claim7}
There is no point of $M$ in the cylinder $\Omega \times \{0 < t \leq h^*\}$ (see Figure \ref{fig:box}).
\end{customthm}

\begin{figure}[htbp]
\centerline{\includegraphics[scale=0.15]{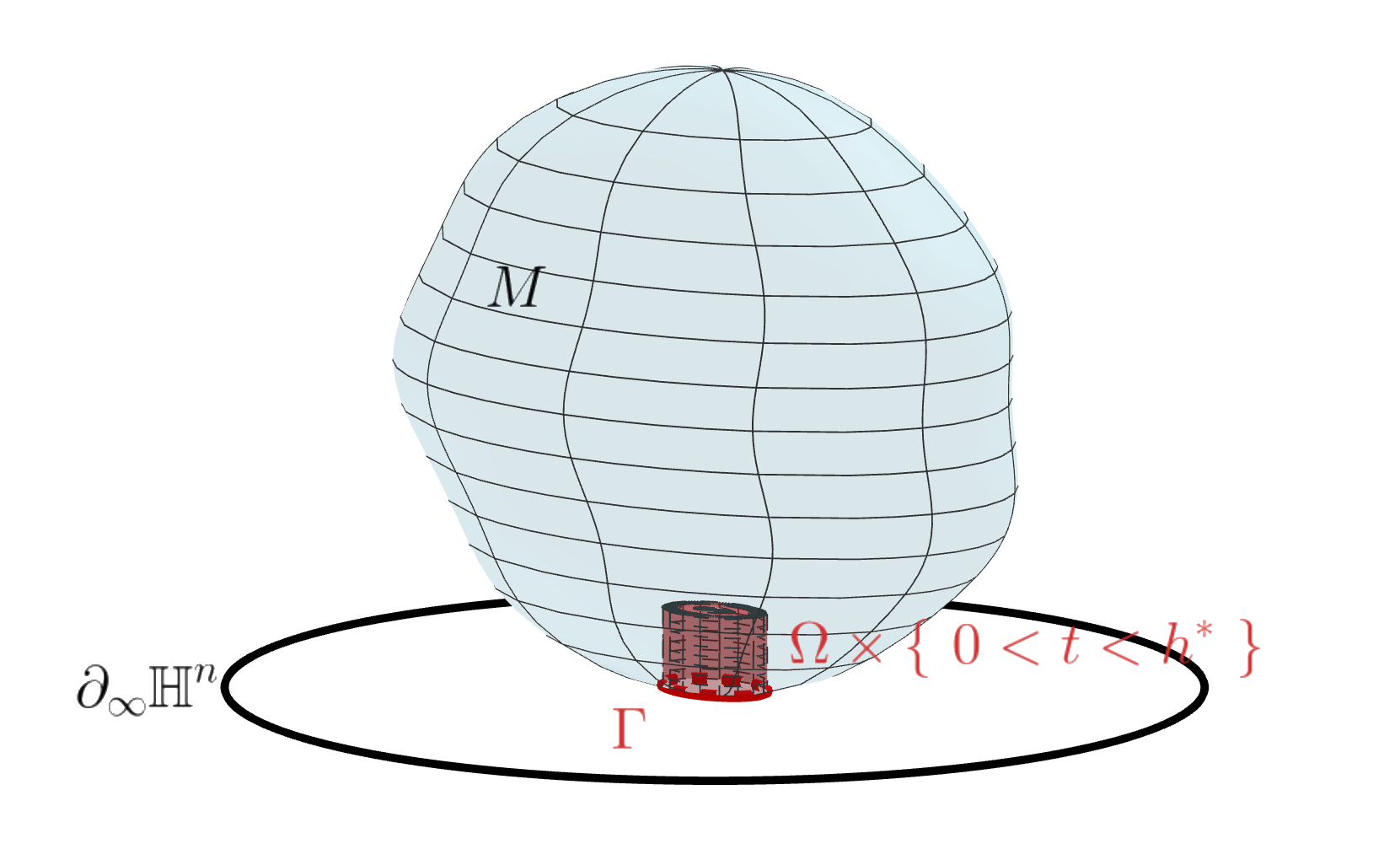}}
\caption{$M$ does not intersect the red cylinder. This final step shows that $M$ has the topology of a disk.}
\label{fig:box}
\end{figure}

\begin{proof}
If $n>r$, $\Sigma$ will denote the portion of the rotational hypersurface generated by the graph of $\lambda_{(n-r)/n,d_r}$  contained in $(D(2\rho_-)\setminus D(\rho_-)) \times \mathbb R$. 
For $n=r$, $\Sigma$ will denote the portion of a peaked sphere generated by the graph of $\lambda_{H_n,d_n}$ contained in $(D(2\rho_-)\setminus D(\rho_-)) \times \mathbb R$.
Note that if $n>r$, then the $r$-th mean curvature of $\Sigma$ is strictly smaller than that of $M$, while if $n=r$ the $n$-th mean curvatures of $M$ and $\Sigma$ coincide.

In both cases $\Gamma$ and  $d_r > 0$ are chosen small enough so that the Claims \ref{claim4}, \ref{claim5}, and \ref{claim6} hold. For any $n\geq r$, $\Sigma$ has two boundary components $C_0$ and $C_1$. Up to vertical translation, we can assume $C_0 \subset \mathbb H^n \times \{0\}$ and $C_1 \subset \mathbb H^n \times \{h^*\}$. Up to horizontal translation we can assume that the center of $C_0$ coincides with the center of the disk of Lemma \ref{lemma5}. Moreover, by definition of $h^*$, the radius of $C_1$ is smaller than $2\rho_-$.

Let $R$ be the radius found in Claim \ref{claim5}. By Claim \ref{claim5}, \eqref{last-estimate} and Lemma \ref{limacon-rmk} we get
$$
R>\ell(\rho^*_{H_r}-r_{ext},r_{ext})>\ell\left(\frac {2\rho^*_{H_r}}{3},\frac{\rho^*_{H_r}}{3}\right)>\frac{\rho^*_{H_r}}{6}.
$$
By Lemma \ref{various-estimates}, we can take $\Gamma$ small enough such that
$$
\frac{\rho^*_{H_r}}{6}>\left(\frac{1}{\alpha}+1\right)\rho_-,
$$ 
where $\alpha$ is the constant in \eqref{alpha}. It follows that if $\Gamma$ is small enough, then
\begin{equation}
\label{stimaR}
R > r_{ext}+\rho_->2\rho_-.
\end{equation}
Claim \ref{claim5} and \eqref{stimaR} allow us to translate $\Sigma$ vertically in such a way that it is contained in $W$.
By  Lemma \ref{lemma5} and the Maximum Principle, we can then translate $\Sigma$ down until $C_0$ reaches $\mathbb H^n \times \{0\}$ without having contact points with the interior of $M$.
Because of $\rho_-<r_{int}$, we can translate horizontally $\Sigma$ in such a way that it touches every point of $\Gamma$ with $C_0$ and keeping $C_0$ inside $\Omega$.

Since \eqref{stimaR} holds true, during this translation $C_1$ remains inside the disk $D^*(R) \subset \mathbb H^n \times \{h^*\}$, which is the reflection of $D(R) \subset \mathbb H^n \times \{h_M-h^*\}$.
By Claim \ref{claim6}, in this process, the upper boundary of $\Sigma$ does not touch $M$. Recalling that the $r$-th mean curvature of $\Sigma$ is not bigger than that of $M$, by the Maximum Principle, we get that there can be no internal contact point between $M$ and $\Sigma$.
The claim then follows because $\Sigma$ is a graph over the exterior of $D(\rho_-)$.
\end{proof}
The proof of Theorem \ref{main-thm} is now complete.
\end{proof}

\section*{List of notations}
\label{appendix}

We include a summary of the various notations we use throughout for the most notable objects and quantities.
\vspace{\topsep}
\begin{enumerate}
\itemsep0.5em 
\item Profile curves:
	\vspace{\topsep}
	\begin{enumerate}
	\itemsep0.2em 
	\item[$\lambda_{H_r,d_r}$:] function defining the profile curve of $H_r$-hypersurfaces in $\mathbb H^n \times \mathbb R$ invariant under rotation depending on a real parameter $d_r$ (Section \ref{rotational-surfaces}).
	\item[$\mu_{H_r,\epsilon}$:] function defining the profile curve of $H_r$-hypersurfaces in $\mathbb H^n \times \mathbb R$ invariant under hyperbolic translation depending on a real parameter $\epsilon > 0$ (Section \ref{translation-surfaces}).
	\end{enumerate}
\item Domain of profile curves:
	\vspace{\topsep}
	\begin{enumerate}
	\itemsep0.2em 
	\item[$\rho_-$:] minimum of the domain of $\lambda_{H_r,d_r}$ when this is not zero.
	\item[$\rho_+$:] maximum of the domain of $\lambda_{H_r,d_r}$.
	\item[$\rho_0$:] minimum point of $\lambda_{H_r,d_r}$ in $(\rho_-,\rho_+)$.
	\item[$\rho_+^{\epsilon}$:] maximum of the domain of $\mu_{H_r,\epsilon}$ for $r > 1$.
	\end{enumerate}
\item Hypersurfaces in $\mathbb H^n \times \mathbb R$:
	\vspace{\topsep}
	\begin{enumerate}
	\itemsep0.2em 
	\item[$\mathcal S_r$:] rotation $H_r$-hypersurface generated by the graph of $\lambda_{H_r,0}$, for some $H_r > (n-r)/n$.
	\item[$\mathcal C_{r,\epsilon}$:] translation $H_r$-hypersurface with $H_r > (n-r)/n$ generated by the graph of $\mu_{H_r,\epsilon}$.
		\end{enumerate}
\item Special quantities:
	\vspace{\topsep}
	\begin{enumerate}
	\itemsep0.2em 
	\item[$R_{\mathcal S_r}$:] the value $\rho_+$ for $\lambda_{H_r,0}$.
	\item[$R_{\mathcal C_{r,\epsilon}}$:] the value $\rho_+^{\epsilon}-\epsilon$ for $r > 1$.
	\item[$h^*$:] approximated value of $\lambda_{H_r,d_r}(2\rho_-)$ for $d_r > 0$ and $H_r = (n-r)/n$ \eqref{h-star}.
	\item[$\rho_{H_r}^*$:] radius of the hypersurface given by the graph of $\lambda_{H_r,0}$, $H_r > (n-r)/n$, at height $h^*$ \eqref{rho*}.
	\end{enumerate}
\item Specific notations for Section \ref{main-result}:
	\vspace{\topsep}
	\begin{enumerate}
	\itemsep0.2em 
	\item[$\mathcal C_r$:] same as $\mathcal C_{r,\epsilon}$ with a choice of $\epsilon$ such that $R_{\mathcal C_{r,\epsilon}} < R_{\mathcal S_r}$.
	\item[$h_{\mathcal C_r}$:] height of $\mathcal C_r$, namely $2\mu_H(\rho_+)$ for $r=1$ and $2\mu_{H_r,\epsilon}(\rho_+^{\epsilon})$ for $r > 1$, cf.~Theorem \ref{str-thm-cyl2}.
	\item[$h_M$:] height of $M \subset \mathbb H^n \times [0,\infty)$ with respect to the  slice $\mathbb H^n \times \{0\}$.
	\item[$\mathcal L$:] the hyperbolic lima\c{c}on as in Definition \ref{limacon-def}.
	\item[$\ell(a,c)$:] optimal radius of a ball bounded by the smaller loop of $\mathcal L$ with parameters $a>c$, see Lemma \ref{limacon-lemma} and identity \eqref{improvement-distance} for its explicit definition.
	\item[$r_{int}$:] interior radius of $\Gamma$.
	\item[$r_{ext}$:] exterior radius of $\Gamma$.
	\item[$r_{min}$:] the largest radius of a ball bounded by the smaller loop of $\mathcal L$ over which $M$ is a graph, see Lemmas \ref{limacon-lemma} and \ref{lemma5}. \\
	\end{enumerate}
\end{enumerate}

\indent All authors declare no conflict of interest. The authors are partially supported by INdAM-GNSAGA. The authors would like to thank the anonymous referees for the careful reading and the valuable suggestions. \\

\end{document}